\documentclass[10pt]{amsart}
\usepackage{amsmath,mathrsfs,amsthm,bbm}
\usepackage{txfonts}
\usepackage[all]{xy}

\usepackage{tikz-cd}
\usetikzlibrary{matrix,arrows,decorations.pathmorphing}

\DeclareMathOperator{\Hom}{\mathrm{Hom}}
\DeclareMathOperator{\End}{\mathrm{End}}
\DeclareMathOperator{\Aut}{\mathrm{Aut}}

\DeclareMathOperator{\SO}{\mathrm{SO}}
\DeclareMathOperator{\GL}{\mathrm{GL}}
\DeclareMathOperator{\SL}{\mathrm{SL}}
\DeclareMathOperator{\tr}{\mathrm{tr}}
\DeclareMathOperator{\OG}{\mathrm{OG}}
\DeclareMathOperator{\IG}{\mathrm{IG}}
\DeclareMathOperator{\Fl}{\mathrm{Fl}}

\newtheorem{theorem}{Theorem}[section]
\newtheorem{lemma}[theorem]{Lemma}

\newtheorem{thm}[theorem]{Theorem}
\newtheorem{prop} [theorem]{Proposition}

\theoremstyle{definition}
\newtheorem{definition}[theorem]{Definition}

\newtheorem{remark}[theorem]{Remark}

\DeclareMathOperator{\poly}{\mathrm{poly}}
\DeclareMathOperator{\Rep}{\mathrm{Rep}}

\DeclareMathOperator{\diag}{\mathrm{diag}}

\DeclareMathOperator{\sym}{\mathrm{sym}}

 \newcommand{\C}{\mathbb{C}}
 \newcommand{\Z}{\mathbb{Z}}
\newcommand{\cl}{\mathcal{L}}

\newcommand{\ft}{\mathfrak{t}}

\newcommand{\rep}{\text{Rep}}

\newcommand{\mf[1]}{\mathfrak{#1}}
\newcommand{\mc[1]}{\mathcal{#1}}

\begin{document}

\title[Representation ring of Levi subgroups versus cohomology]
{Representation ring of Levi subgroups versus cohomology ring of flag varieties III}

\author{Shrawan Kumar and Jiale Xie}
\maketitle

\begin{abstract} For any reductive group $G$ and a parabolic subgroup $P$ with its Levi subgroup $L$, the first author  [Ku2] introduced a ring homomorphism 
$ \xi^P_\lambda:  \Rep^\C_{\lambda-\poly}(L) \to H^*(G/P, \C)$, where  $ \Rep^\C_{\lambda-\poly}(L)$ is a certain subring of the complexified representation ring of $L$ (depending upon the choice of an irreducible representation $V(\lambda)$ of $G$ with highest weight $\lambda$). In this paper we study this homomorphism for $G=\SO(2n)$ and its maximal parabolic subgroups $P_{n-k}$ for any $2\leq k\leq n-1$ (with the choice of $V(\lambda) $ to be the defining representation  $V(\omega_1) $ in $\mathbb{C}^{2n}$). Thus, we  obtain a $\C$-algebra homomorphism  $ \xi_{n,k}^D:  \Rep^\C_{\omega_1-\poly}(\SO(2k)) \to H^*(\OG(n-k, 2n), \C)$. We determine this homomorphism explicitly in the paper. We further analyze the behavior of $ \xi_{n,k}^D$ when $n$ tends to $\infty$ keeping $k$ fixed and show 
 that $ \xi_{n,k}$ becomes injective in the limit. We also determine explicitly (via some computer calculation) the homomorphism $ \xi^P_\lambda$ 
 for all the exceptional groups $G$  (with a specific `minimal'  choice of $\lambda$)  and  all their maximal parabolic subgroups except $E_8$. 
 
\end{abstract}

\section{Introduction}

This is a follow-up of the papers \cite{Ku2} and  \cite{KR}.

Let $G$ be a connected reductive algebraic group over $\C$ with Borel subgroup $B$ and maximal torus $T\subset B$. Let $P$ be a standard parabolic subgroup with Levi subgroup $L$ containing $T$. Let $W$ (resp. $W_L$) be the Weyl group of $G$ (resp. $L$). Let $V(\lambda)$ be an irreducible almost faithful representation of $G$ with highest weight $\lambda$, i.e.,  $\lambda$ is a dominant integral weight and the corresponding map $\rho_\lambda:G\to \GL(V(\lambda))$ has finite kernel. Then, Springer  defined an adjoint-equivariant regular map with Zariski dense image from the group to its Lie algebra, $\theta_\lambda:G\to \mf[g]$, which depends on $\lambda$ (cf. \cite{BR}). 

Using the Springer morphism  $\theta_\lambda$, Kumar \cite{Ku2} defined the $\lambda$-polynomial representation ring $\rep^\C_{\lambda-\text{poly}}(L)$, which is a subring of the complexified representation ring $\rep^\C(L)$. For $G=\GL(n)$ and $V(\lambda)$ the defining representation $V(\omega_1)=\C^n$, the ring  $\Rep_{\omega_1-\poly}(G):=  \Rep^\C_{\omega_1-\poly}(G)\cap  \Rep (G)$
coincides with the standard  polynomial representation ring $\Rep_{\poly}(G)$ of $\GL(n)$. 

Kumar  \cite{Ku2}
also defined a surjective $\C$-algebra homomorphism 
\begin{equation}\label{xi p morphism}
    \xi^P_\lambda: \rep^\C_{\lambda-\text{poly}}(L)\to H^*(G/P,\C).
\end{equation}
This map generalizes (cf. \cite[Theorem 8]{Ku2}) the classical ring homomorphism 
$$\phi_{n, r}:\rep_{\text{poly}}(\text{GL}(r))\to H^*(\text{Gr}(r,n), \Z),$$
where  $\text{Gr}(r,n)$ is the Grassmannian of $r$-planes in $\C^n$, which is also the quotient $\text{GL}(n)/P_r$  for the maximal parabolic subgroup $P_r$ of $\text{GL}(n)$ obtained from deleting the $r$-th node in the Dynkin diagram of 
$\SL(n)$.

In a subsequent work, Kumar--Rogers \cite{KR} explicitly determined the  homormorphism $\xi^P_{\omega_1}$ for the classical groups $G$ of types $B$ and $C$ and the maximal parabolic subgroups, where $\omega_1$ is the first fundamental weight. Moreover, in these cases, they further analysed the map $\xi^P_{\omega_1}$ in the limit described below:

For type $C_n$, i.e.,  $G=\text{Sp}(2n)$ (the symplectic group) and  any positive integer $0< k<n$, consider the  standard maximal parabolic subgroup $P^C_{n-k}$ corresponding to the $(n-k)$-th node  of the Dynkin diagram of $\text{Sp}(2n)$, and let $L^C_{n-k}$ be its  Levi subgroup. Then, the quotient $\text{Sp}(2n)/P^C_{n-k}$ is  the isotropic Grassmannian $\IG(n-k,2n)$, and we have 
$$L^C_{n-k}\simeq \text{GL}(n-k)\times \text{Sp}(2k).$$
Thus, following \eqref{xi p morphism}, we get a ring homomorphism 
$$\xi^{P^C_{n-k}}_{\omega_1}:\rep^\C_{\omega_1-\text{poly}}(L^C_{n-k})\to H^*(\text{IG}(n-k,2n),\C).$$
Restricting $\xi^{P^C_{n-k}}_{\omega_1}$ to the component $\text{Sp}(2k)$, we get a ring homomorphism 
$$\xi^C_{n,k}:\rep^\C_{\omega_1-\text{poly}}(\text{Sp}(2k))\to H^*(\text{IG}(n-k,2n),\C).$$
Define the stable cohomology ring
$$\mathbb{H}^*(\text{IG}_k,\C):=\varprojlim_nH^*(\text{IG}(n-k,2n),\C)$$
as the inverse limit under a certain embedding of $\IG(n-k, 2n) \hookrightarrow \IG(n+1-k, 2(n+1))$ (cf. \cite[Definition 15]{KR}). Then, the homomorphisms $(\xi^C_{n,k})_{n>k}$ combine to give a ring homomorphism $$\xi^C_k:\rep^\C_{\omega_1-\text{poly}}(\text{Sp}(2k))\to \mathbb{H}^*(\text{IG}_k,\C).$$
Kumar--Rogers proved that the homomorphism $\xi_k^C$ is injective \cite[Theorem 16]{KR} (however it is not surjective). They also obtained parallel results for the classical groups of type $B$.
\vskip1ex

\emph{The aim of this paper is to complete the above results for the even orthogonal groups $\text{SO}(2n)$, and also the exceptional  groups of type $G_2, F_4, E_6$  and $E_7$.  In principle, we can also handle $E_8$, but the results in this case are too long and complicated to reproduce here.}
\vskip1ex

For $n\ge 4$ and $2\le k\le n-1$,  let $\text{OG}(n-k,2n)$ be the set of $(n-k)$-dimensional isotropic subspaces in $V=\C^{2n}$. Then, $\text{OG}(n-k,2n)$ is the quotient $\text{SO}(2n)/P^D_{n-k}$ by the standrad maximal parabolic subgroup $P^D_{n-k}$ corresponding to the $(n-k)$-th node of the Dynkin diagram of type $D_n$. Let $L^D_{n-k}$ be its Levi subgroup. Then,
$$L^D_{n-k}\simeq \GL(n-k)\times \SO(2k).$$
Thus, following \eqref{xi p morphism}, we get a ring homomorphism 
$$\xi^{P^D_{n-k}}_{\omega_1}:\rep^\C_{\omega_1-\text{poly}}(L^D_{n-k})\to H^*(\text{OG}(n-k,2n),\C).$$
Restricting $\xi^{P^D_{n-k}}_{\omega_1}$ to the component $\text{SO}(2k)$, we get a ring homomorphism 
$$\xi^D_{n,k}:\rep^\C_{\omega_1-\text{poly}}(\text{SO}(2k))\to H^*(\text{OG}(n-k,2n),\C).$$
We determine this ring homomorphism explicitly in Theorem \ref{xi image D}.

Define the stable cohomology ring
$$\mathbb{H}^*(\text{OG}_k,\C):=\varprojlim_nH^*(\text{OG}(n-k,2n),\C)$$
as the inverse limit under the embedding $\pi_n: \OG(n-k,2n)\hookrightarrow \OG(n+1-k,2(n+1))$
(cf. the discussion after \eqref{inverse system D}). 
 Then, the homomorphisms $(\xi^D_{n,k})_{n\geq k+2}$ combine to give a ring homomorphism $$\xi^D_k:\rep^\C_{\omega_1-\text{poly}}(\text{SO}(2k))\to \mathbb{H}^*(\text{OG}_k,\C).$$
Following is our one of the  main results of the paper (cf. Theorem \ref{injectivity} and Remark \ref{newremark}).
\begin{thm}
 The above ring homomorphism $\xi^D_k:\rep^\C_{\omega_1-\text{poly}}(\text{SO}(2k))\to \mathbb{H}^*(\text{OG}_k,\C)$ is injective.
 However, it is not surjective.
\end{thm}
For  the exceptional groups of type $G_2, F_4, E_6$  and $E_7$,
 we  calculate the Springer morphism $\theta_\lambda$ restricted to $T$ by using \cite[Theorems 1, 2]{R}, where we take $\lambda$ such that $V(\lambda)$ has the minimum Dynkin index. Further, for these exceptional groups and any maximal parabolic subgroup $P$ we determine {\it explicitly} the generators of 
  $\rep^\C_{\lambda-\poly}(L)$ and their images in $H^*(G/P, \C)$ under $\xi^P_\lambda$ in the Schubert basis (cf. Sections 5-8), using a computer program \cite{X}, which is developed based on the algorithms in \cite{blee} and \cite{D}.

The proofs in the case of $\SO(2n)$ rely on some results of Buch-Kresch-Tamvakis  \cite{BKT1},  work of Kumar \cite{Ku2} and work of  Kumar-Rogers \cite{KR}.    
  
\vskip2ex
\noindent
{\bf Acknowledgements:} The first author was partially supported by the NSF grant DMS-1802328. We thank Leonardo Mihalcea, Alex Molev, Sean Rogers and Harry Tamvakis for some helpful correspondences.

\section{Prelimanaries and Notation}

We recall some notation and results from [Ku2]. 

Let $G$ be a connected reductive group  over $\mathbb{C}$ with a Borel subgroup $B$ and maximal torus $T\subset B$. Let $P$ be a standard parabolic subgroup with the Levi subgroup $L$ containing $T$. We denote their Lie algebras by the corresponding Gothic characters: $\mathfrak{g}, \mathfrak{b}, \mathfrak{t},\mathfrak{p}, \mathfrak{l}$
respectively. We denote by $\Delta=\{\alpha_1, \dots, \alpha_\ell\}\subset \mathfrak{t}^*$ the set of simple roots. The fundamental weights of $\mathfrak{g}$ are denoted by  $\{\omega_1, \dots, \omega_\ell\}\subset \mathfrak{t}^*$. Let $W$ (resp. $W_L$) be the Weyl group of $G$ (resp. $L$).
Then, $W$ is generated by the simple reflections $\{s_i\}_{1\leq i \leq \ell}$. Let $W^P$ denote the set of smallest coset representatives in the cosets in $W/W_L$. {\it Throughout the paper we follow the indexing convention as in [Bo, Planche
I - IX].} 

Let $X(T)$ be the group of characters of $T$ and let $D\subset X(T)$ be the set of dominant characters  (with respect to the given choice of $B$ and hence positive roots, which are the roots of $\mathfrak{b}$). Then, the isomorphism classes of  finite dimensional irreducible representations of $G$ are bijectively parameterized by $D$ under the correspondence $\lambda \in D \leadsto V(\lambda)$, where
$V(\lambda)$ is the irreducible representation of $G$ with highest weight $\lambda$. We call $V(\lambda)$ {\it almost faithful} if the corresponding map $\rho_\lambda: G \to \Aut (V(\lambda))$ has finite kernel. 

Recall the 
Bruhat decomposition for the flag variety:
$$G/P=\sqcup_{w\in W^P}\, \Lambda_w^P,\,\,\,\text{where}\,\, \Lambda^P_w:= BwP/P.$$
Let $\bar{\Lambda}_w^P$ denote the closure of $\Lambda_w^P$ in $G/P$.  We denote by 
$[\bar{\Lambda}_w^P] \in H_{2 \ell(w)}(G/P, \mathbb{Z})$ its
fundamental class. Let $\{\epsilon^P_w\}_{w\in W^P}$ denote the Kronecker dual basis of the cohomology, i.e., 
$$\epsilon^P_w([\bar{\Lambda}_v^P])= \delta_{w,v}, \,\,\,\text{for any}\,\, v,w\in W^P.$$
Thus, $\epsilon^P_w$ belongs to the singular cohomology:
$$\epsilon^P_w\in H^{2 \ell(w)}(G/P, \mathbb{Z}).$$
We abbreviate $\epsilon^B_w$ by $\epsilon_w$. Then, for any $w\in W^P, \epsilon^P_w=\pi^*(\epsilon_w)$, where $\pi:G/B\to G/P$ is the standard projection.

\begin{definition}\label{def1}
Let $V(\lambda)$ be any almost faithful irreducible representation of $G$. Following Springer (cf. [BR, $\S$9]), define the map
$$
\theta_{\lambda}:G\to \mathfrak{g}\quad \text{(depending upon $\lambda$)}
$$
as follows:
\[
\xymatrix{
G\ar[r]^-{\rho_{\lambda}}\ar[dr]^{\theta_{\lambda}} & \Aut (V(\lambda))\subset \End (V(\lambda))=\mathfrak{g}\oplus \mathfrak{g}^{\perp}\ar[d]^{\pi}\\
 & \mathfrak{g}
}
\]
where  $\mathfrak{g}$ sits canonically inside $\End(V(\lambda))$ via the derivative $d\rho_{\lambda}$, the orthogonal complement $\mathfrak{g}^{\perp}$
is  taken with respect to the standard conjugate $ \Aut(V(\lambda))$-invariant form on $\End (V(\lambda))$: $\langle A, B\rangle :=\tr (AB)$,  and $\pi$ is the projection to the $\mathfrak{g}$-factor. (By considering a compact form $K$ of $G$, it is easy to see that  $ \mathfrak{g} \cap \mathfrak{g}^{\perp} =\{0\}$.)

Since $\pi\circ d\rho_{\lambda}$ is the identity map, $\theta_{\lambda}$ is a local diffeomorphism at $1$ (and hence with Zariski dense image). Of course, by construction, $\theta_{\lambda}$ is an algebraic morphism. Moreover, since the decomposition $ \End (V(\lambda))=\mathfrak{g}\oplus \mathfrak{g}^{\perp}$ is $G$-stable, it is easy to see that $\theta_\lambda$ is $G$-equivariant under conjugation.
\end{definition}
We recall the following lemma from [Ku2, Lemma 2].
\begin{lemma} \label{lemma1} The above morphism restricts to ${\theta_\lambda}_{|T}: T \to \mathfrak{t}$.
\end{lemma}

 For any $\mu \in X(T)$, we have a  $G$-equivariant
 line bundle $\cl (\mu)$  on $G/B$ associated to the principal $B$-bundle $G\to G/B$
via the one dimensional $B$-module $\mu^{-1}$. (Any  $\mu \in X(T)$ extends
 uniquely to a character of $B$.) The one dimensional $B$-module $\mu$ is also denoted by
 $\mathbb{C}_\mu$. Recall the surjective  Borel homomorphism
 $$\beta : S(\mathfrak{t}^*) \to H^*(G/B, \mathbb{C}),$$
which takes a character $ \mu \in X(T)$ to the first Chern class of the  line bundle $\cl(\mu)$. (We realize 
$X(T)$ as a lattice in $\mathfrak{t}^*$ via taking derivative.) We then extend this map linearly over $\mathbb{C}$ to $\mathfrak{t}^*$ and extend further as a graded algebra homomorphism from $ S(\mathfrak{t}^*)$ (doubling the degree). Under the Borel homomorphism,
\begin{equation}\label{eqnborel} \beta(\omega_i)=\epsilon_{s_i},\,\,\,\text{for any fundamental weight}\,\, \omega_i.
\end{equation}

Fix a compact form $K$ of $G$ such that $T_o:=K\cap T$ is a (compact) maximal torus of $K$. Then, $W\simeq N(T_o)/T_o$, where $N(T_o)$ is the normalizer of $T_o$ in $K$. Recall that $\beta$ is $W$-equivariant under  the standard action of $W$ on  $S(\mathfrak{t}^*)$ and the $W$-action on $H^*(G/B, \mathbb{C})$ induced from the $W$-action on $G/B\simeq K/T_o$ via 
$$(nT_o)\cdot (kT_o):= kn^{-1}T_o,\,\,\,\text{for}\,\, n\in N(T_o)\,\,\,\text{and}\,\,  k\in K.$$
Thus, for any standard parabolic subgroup $P$ with the Levi subgroup $L$ containing $T$, restricting $\beta$, we get a surjective graded algebra homorphism:
$$\beta^P : S(\mathfrak{t}^*)^{W_L} \to H^*(G/B, \mathbb{C})^{W_L}\simeq  H^*(G/P, \mathbb{C}),$$
where the last isomorphism, which is induced from the projection $G/B \to G/P$,  can be found, e.g.,  in [Ku1, Corollary 11.3.14]. 

Now, the Springer morphism ${\theta_\lambda}_{|T}:T\to \mathfrak{t}$ (restricted to $T$) gives rise to the corresponding $W$-equivariant injective  algebra homomorphism 
on the affine coordinate rings:
$$ ({\theta_\lambda}_{|T})^*: \mathbb{C}[\mathfrak{t}]=S(\mathfrak{t}^*)\to  \mathbb{C}[T].$$
Thus,  on restriction to $W_L$-invariants, we get an injective algebra homomorphism 
$$ {\theta_\lambda (P)}^*: \mathbb{C}[\mathfrak{t}]^{W_L}=S(\mathfrak{t}^*)^{W_L}\to  \mathbb{C}[T]^{W_L}.$$
(Since $W_L$-invariants depend upon the choice of the parabolic subgroup $P$, we have included $P$ in the notation of $ {\theta_\lambda (P)}^*$.) Now, let $\Rep (L)$ be the representation ring of $L$ and let  $\Rep^\mathbb{C} (L):= \Rep (L)\otimes_{\mathbb{Z}}\,\mathbb{C}$ be its complexification.  Then, as it is well known,   
\begin{equation} \label{eqn1}\Rep^\mathbb{C} (L)\simeq  \mathbb{C}[T]^{W_L}
\end{equation}
obtained from taking the character of an $L$-module restricted to $T$. 

\begin{definition} \label{maindefi} We  call a virtual character $\chi\in  \Rep^\mathbb{C} (L)$ of $L$ a {\it $\lambda$-polynomial character} if the corresponding function in 
$ \mathbb{C}[T]^{W_L}$ is in the image of $ {\theta_\lambda (P)}^*$. The set of all $\lambda$-polynomial characters of $L$, which is, by definition,  a subalgebra of $\Rep^\mathbb{C} (L)$ isomorphic to the algebra $S(\mathfrak{t}^*)^{W_L}$, is denoted by $\Rep^\mathbb{C}_{\lambda-\poly} (L)$. Of course, the map $ {\theta_\lambda (P)}^*$ induces an algebra isomorphism (still denoted by)
$$ {\theta_\lambda (P)}^*: S(\mathfrak{t}^*)^{W_L}\simeq \Rep^\mathbb{C}_{\lambda-\poly} (L),$$
under the identification \eqref{eqn1}.
\end{definition}

It is easy to see that 
\begin{equation} \label{eqn1'} \Rep_{\omega_1-\poly}(\GL (n))= \Rep_{\poly}(\GL (n)),
\end{equation}
where  $ \Rep_{\poly}(\GL (n))$ denotes the subring of the representation ring $\Rep(\GL (n))$ spanned by the irreducible polynomial representations of $\GL (n)$. 

We recall  the following  result from [Ku2, Theorem 5].

\begin{theorem} \label{thmmain}  Let $V(\lambda)$ be an almost faithful irreducible $G$-module and let $P$ be any standard parabolic subgroup. Then,  
the above maps (specifically $\beta^P\circ ({\theta_\lambda (P)}^*)^{-1} $) give
 rise to  a surjective $\mathbb{C}$-algebra homomorphism 
$$\xi_\lambda^P: \Rep^\mathbb{C}_{\lambda-\poly}(L)  \to H^*(G/P, \C).$$

Moreover, let $Q$ be another standard  parabolic subgroup with Levi subgroup $R$ containing $T$ such that $P\subset Q$ (and hence 
$L\subset R$). Then, we have the following commutative diagram:\[
\xymatrix{
\Rep^\mathbb{C}_{\lambda-\poly}(R)  \ar[d]^{\gamma}\ar[r]^{\xi_\lambda^Q} & H^*(G/Q, \C)
\ar[d]^{\pi^*}\\
\Rep^\mathbb{C}_{\lambda-\poly}(L)\ar[r]^{\xi_\lambda^P} &H^*(G/P, \C),
}
\]
where $\pi^*$ is induced from the standard projection $\pi:G/P \to G/Q$ and $\gamma$ is induced from the restriction of representations. 
\end{theorem}

\section{Injectivity Result for the even orthogonal Groups}

{\it In this section, we take $n\geq 4$} to avoid some trivial cases.
\vskip1ex

Let $V=\C^{2n}$ be equipped with the non-degenerate quadratic form $\langle\,,\,\rangle$ so that its matrix $(\langle e_i,e_j\rangle )_{1\leq i, j\leq 2n}$ in the standard basis $\{e_1, \dots, e_{2n}\}$ is given by the antidiagonal matrix $E_D$ with all its antidiagonal entries equal to $1$. Then,
 $$\SO(2n)=\{g\in \SL (2n):(g^t)^{-1}=E_DgE_D^{-1}\}$$
 is the special orthogonal group. Now, as in \cite[Lemma 10]{Ku2}, 
the Springer morphism $\theta = \theta_{\omega_1}$ (with respect to the first fundamental weight $\omega_1$) for $G= \SO(2n)$ is given by the Cayley Transform: 
$$g\to \frac{g-E_D^{-1}g^tE_D}{2},\,\,\,\text{ for $g\in G$}.$$
Take the maximal torus in $\SO(2n)$ defined by
\begin{equation}\label{tori}
    T_D=\{\mathbf{t}=\diag(t_1,\ldots,t_n,t_n^{-1},\ldots,t_1^{-1}):t_i\in \C^*\}.
\end{equation}
Its Lie algebra is given by
\begin{equation}\label{tori algebra}
    \mf[t]_D=\{\overline{\mathbf{t}}=\diag(x_1,\ldots,x_n,-x_n,\ldots,-x_1):x_i\in \C\}.
\end{equation}
Restrict the Springer morphism $\theta$ to the torus to get
\begin{equation}\label{theta t}
\theta(\mathbf{t})=\diag(\overline{t}_1,\ldots,\overline{t}_n,-\overline{t}_n,\ldots,-\overline{t}_1),\,\,\,\text{
where $\overline{t}_i :=\frac{t_i-t_i^{-1}}{2}$}.
\end{equation}
The Weyl group $W_G$ of $G$ is a subgroup of $S_{2n}$ consisting of permutations $(a_1\cdots a_{2n})$ satisfying $a_i+a_{2n+1-i}=2n+1$ for $1\le i\le 2n$, and the cardinality of the set $\{i\mid 1\le i\le n,\ a_i>n\}$ is even (cf. \cite[$\S$3.5]{BL}).
Recall the following result from \cite[Proposition 12]{Ku2}.
\begin{prop}\label{rep ring D}
 Let $f:T\to \C$ be a regular map. Then, $f\in \Rep_{\poly}^\C(G) = \Rep_{\omega_1-\poly}^\C(G) 
 $ if and only if there exist symmetric polynomials $P_f(x_1,\ldots,x_n)$ and $Q_f(x_1,\ldots,x_n)$ satisfying
 \begin{equation}
     f(\mathbf{t})=P_f(\overline{t}_1^2,\ldots,\overline{t}_n^2)+\left(\overline{t}_1\overline{t}_2 \cdots \overline{t}_n\right)Q_f(\overline{t}^2_1,\ldots,\overline{t}^2_n)
 \end{equation}
\end{prop}
where $\mathbf{t}$ is from \eqref{tori} and $\overline{t}_i$ from \eqref{theta t}.\\

\begin{definition}
 For $1\le r\le n-2$, we let $\OG=\OG(r,2n)$ to be the set of $r$-dimensional isotropic subspaces in $V=\C^{2n}$. We take $B_D := B\cap \SO(2n)$, where $B$ is the Borel subgroup of $\SL(2n)$ consisting of upper triangular matrices (of determinant $1$).
Then, $\OG(r,2n)$ is the quotient $\SO(2n)/P_r^D$ by the standard maximal parabolic subgroup $P_r^D$ with $\Delta\setminus\{\alpha_r\}$ as the set of simple roots of its Levi component $L_r^D$. Here $L_r^D$ is the unique Levi subgroup of $P_r^D$ containing $T_D$. We have
\begin{equation}\label{levi D}
    L_r^D\simeq \GL(r)\times \SO(2(n-r)).
\end{equation}
By \cite[$\S$9]{Ku2},
\begin{equation}\label{rep ring D iso}
    \Rep_{\poly}^\C(L_r^D)\simeq \C_{\sym}[\overline{t}_1,\ldots,\overline{t}_r] \otimes_\C  \left( R\oplus \overline{t}_{r+1}\cdots \overline{t}_n R \right), \,\,\text{ where $R=\C_{\sym}[\overline{t}_{r+1}^2,\ldots,\overline{t}_n^2]$}.\end{equation}
Since the case of $r=n-1$ is parallel to $r=n$ due to the diagram automorphism of $D_n$  taking the $n$-th node  to the $(n-1)$-th node, we only consider $r=n$ here. We have $\SO(2n)/P_n^D$ identified with a connected component $\OG(n,2n)_+$ of the set of $n$-dimensional isotropic subspaces of $V$. Moreover, its Levi subgroup is given by
\begin{equation}\label{levi D n}
    L_n^D\simeq \GL(n).
\end{equation}
In this case, by \cite[$\S$9]{Ku2},
\begin{equation}\label{rep ring D n iso}
    \Rep_{\poly}^\C(L_n^D)\simeq \C_{\sym}[\overline{t}_1,\ldots,\overline{t}_n].
\end{equation}
\end{definition}

Now,  we describe the geometry of $\OG(r,2n)$ following \cite{BKT1}. Two subspaces $E$ and $F$ of $V$ are said to be in the {\it same family} if $\dim(E\cap F)\equiv n (\text{mod } 2)$. Fix once and for all an isotropic subspace $L_n$ of $V$ of dimension $n$. An isotropic flag is a complete flag $F_\bullet$ consisting of subspaces $F_i$ of $V$, with $0=F_0\subsetneq F_1 \subsetneq F_2 \subsetneq\cdots \subsetneq F_{2n}=V$, such that $F_{n+i}=F_{n-i}^\perp$ for all $0\le i\le n-1$, and with $F_n$ and $L_n$ in the same family. As the orthogonal space $F_{n-1}^\perp/F_{n-1}$ contains only two isotropic lines, to each such flag $F_\bullet$ there corresponds an alternate isotropic flag $\Tilde{F}_\bullet$ such that $\Tilde{F}_i= F_i$ for all $i<n$ but with $\Tilde{F}_n$ in the opposite family from $F_n$.

Let $\Fl=\Fl_n=\SO(2n)/B_D$ be the full flag variety consisting of the isotropic flags. Then, the flags $F_\bullet$ give rise to a sequence of tautological vector bundles over $\Fl$: $$\mc[F]_1\subset \mc[F]_2\subset\cdots \subset \mc[F]_n\subset \mc[E]\text{, with }rank(\mc[F]_i)=i,$$ where $\mc[E]:\Fl\times \C^{2n}\to\Fl$ is the trivial rank-$2n$ vector bundle.\\

Set $k=n-r\geq 2$, the Schubert varieties in $\OG(r,2n)$ are indexed by a set $\Tilde{\mc[P]}(k,n)$ defined as follows. First for any non-negative integer $k$ we say a partition $\lambda$ is {\it $k$-strict} if no part greater than $k$ is repeated, i.e., $\lambda_j>k$ implies $\lambda_{j+1}<\lambda_j$. To any $k$-strict partition $\lambda$, there is associated a number in $\{0,1,2\}$ called the {\it type} of $\lambda$, denoted $\text{type}(\lambda)$ (cf. \cite[$\S$3]{BKT1}). For example, if $\lambda$ has no part equal to $k$, then $\text{type}(\lambda)=0$. The elements of $\Tilde{\mc[P]}(k,n)$ are the $k$-strict partitions contained in an $(n-k)\times(n+k-1)$ rectangle. For any $\lambda\in \Tilde{\mc[P]}(k,n)$, we define index set $P'=P'(\lambda) =\{p'_1(\lambda)<\cdots <p'_r(\lambda)\}\subset[1,2n]$ by
\begin{align*}
    p'_j(\lambda) =& n+k-1-\lambda_j+\#\{i<j\mid \lambda_i+\lambda_j\le 2k-1+j-i\} \\
    +& \begin{cases}
      1, & \text{if }\lambda_j>k\text{, or }\lambda_j=k<\lambda_{j-1}\text{ and }n-1+j+\text{type}(\lambda)\text{ is even,} \\
      2, & \text{otherwise.}
    \end{cases}
\end{align*}
Given any fixed isotropic flag $F_\bullet$, each $\lambda \in \Tilde{\mc[P]}(k,n)$ indexes a codimension $|\lambda|$ Schubert variety $X_\lambda(F_\bullet)$ in $\OG(r,2n)$, defined as the locus of $\Sigma\in \OG(r,2n)$ such that
$$\dim(\Sigma\cap F_{p_j'(\lambda)})\ge j\text{, if }p_j'(\lambda)\ne n+1,$$
$$\dim(\Sigma\cap \Tilde{F}_{n})\ge j\text{, if }p_j'(\lambda)= n+1$$
for all $j$ with $1\le j\le \ell(\lambda)$, where $\ell(\lambda)$ is the number of non-zero parts of $\lambda$. For each $\lambda\in \Tilde{\mc[P]}(k,n)$, we let $\tau_\lambda$ denote the cohomology class in $H^{2|\lambda|}(\OG(r,2n),\Z)$ dual to the cycle determined by the Schubert variety indexed by $\lambda$.
\vskip1ex

The special Schubert varieties for $\OG(r, 2n)$ are $X_1,\ldots,X_{k},X_k',X_{k+1},\ldots,X_{n+k-1}$. They are defined by a single Schubert condition as follows. For $p\ne k$, we have
$$X_p(F_\bullet)=\{\Sigma\in \OG(r, 2n)\mid \Sigma\cap F_{\epsilon(p)}\ne 0\}$$
where 
  \begin{equation}
    \epsilon(p) = n+k-p+
    \begin{cases}
      2, & \text{ if }p\le k \\
      1, & \text{ if }p>k.
    \end{cases}
  \end{equation}
For any  $n$, define  $$X_k(F_\bullet)=\{\Sigma\in \OG(r, 2n)\mid \Sigma\cap F_{n}\ne 0\}$$ and $$X_k'(F_\bullet)=\{\Sigma\in \OG(r, 2n)\mid \Sigma\cap \Tilde{F}_{n}\ne 0\}.$$
 We let $\tau_p$ denote the cohomology class of $X_p(F_\bullet)$ for $1\le p\le n+k-1$ and $\tau_k'$ denote the cohomology class of $X_k'(F_\bullet)$. Our definition of $\tau_k, \tau_k'$ is slightly different from that of \cite[$\S$3.2]{BKT1} to make them compatible under taking limits as $n \to \infty$ (cf. the proof of Proposition \ref{stable cohomology presentation}). Their $\tau_k, \tau_k'$ coincides with our  $\tau_k, \tau_k'$ respectively for odd $n$. However, for even $n$, their  $\tau_k'$ (resp. $\tau_k$) is our $\tau_k$ (resp. $\tau_k'$).

We have a short exact sequence of vector bundles on $\OG(r=n-k,2n)$. 
\begin{equation}\label{ses_grmn}
    0\to \mathcal{S}\to V_D\to \mc[Q]\to 0.
\end{equation}
Here $V_D$ is the trivial bundle of rank $2n$, $\mathcal{S}$ is the taotological subbundle of rank $n-k$, and $Q$ is the quotient bundle of rank $n+k$. Let $c_i=c_i(Q)$ for $1\le i\le n+k$ be the $i$-th Chern class of the quotient bundle $Q$. 
Then, we have the following presentation of $H^*(\OG(r,2n),\Z)$ from \cite[Theorem 3.2]{BKT1}.

 First, for each $s>0$, let $\Delta_s$ denote the $s\times s$ Schur determinant:
  $$\Delta_s=\det(c_{1+j-i})_{1\le i,j\le s}.$$  
  By \cite[Identity 40]{BKT1},  we have
\begin{equation}\label{c_p describe}
    c_p( \mc[Q] ) = 
    \begin{cases}
      \tau_p, & \text{ if }p< k, \\
      \tau_k+\tau_k', & \text{ if }p=k,\\
      2\tau_p , & \text{ if }p>k.
    \end{cases}
 \end{equation}  
  Then, we have:
\begin{thm}\label{presentation cohomology D}
 The cohomology ring $H^*(\OG(n-k,2n),\Z)$ is presented as a quotient of the polynomial ring $$\Z[\tau_1,\ldots,\tau_k,\tau_k',\tau_{k+1},\ldots,\tau_{n+k-1}]$$ modulo the relations
 \begin{equation}\label{presentation T1}
     \Delta_s=0\text{, }n-k<s <n,
 \end{equation}
 \begin{equation}\label{presentation T2}
     \tau_k\Delta_{n-k} = \tau_k'\Delta_{n-k} = \sum_{p=k+1}^n(-1)^{p+k+1}\tau_p\Delta_{n-p},
 \end{equation}
 \begin{equation}\label{presentation T3}
     \sum_{p=k+1}^s(-1)^{p}\tau_p\Delta_{s-p}=0\text{, }n<s<n+k,
 \end{equation}
 \begin{equation}\label{presentation S1}
     \tau_s^2+\sum_{p=1}^s(-1)^p\tau_{s+p}c_{s-p}=0\text{, }k+1\le s<n,
 \end{equation}
 \begin{equation}\label{presentation S2}
     \tau_k\tau_k'+\sum_{p=1}^k(-1)^p\tau_{k+p}\tau_{k-p}=0,
 \end{equation}
 where we set $\tau_0=1, \tau_p=0$ for $p\geq n+k$. 
\end{thm}

Now,  we can get a similar result as   \cite[Proposition 11]{KR} for the image of $\xi^P$. First, we define $$x_j=-c_1(\mathcal{F}_j/\mathcal{F}_{j-1}),\,\,\,\text{ for $1\le j\le n$}. $$
\begin{lemma}\label{epsilon s_j}
  For $1\le j\le n$, the Schubert divisor $\epsilon_{s_j}\in H^2(\Fl,\Z)$ corresponding to the simple reflection $s_j$ is given by \begin{align*} \epsilon_{s_j}&=x_1+\cdots + x_j,\,\,\,\text{for $1\leq j\leq n-2$}\\
  &=\frac{1}{2}(x_1+\cdots + x_{n-1}-x_n),\,\,\,\text{for $j= n-1$}\\  
  &=\frac{1}{2}(x_1+\cdots + x_n),\,\,\,\text{for $ j=n$}.
  \end{align*} 
  Moreover, we have $\xi^B(\overline{t}_j)=x_j$\, for all $1\le j\le n$.
\end{lemma}
\begin{proof} By the identity \eqref{eqnborel}, for any $1\leq j\leq n$,
\begin{equation} \label{neweqn1}
\beta (\omega_j) :=c_1(\mathcal{L}(\omega_j))= \epsilon_{s_j},
\end{equation}
where $\beta: S(\mathfrak{t}^*_D) \to H^*(\Fl, \mathbb{C})$ is the Borel homomorphism.

For $1\leq i\leq n$, since $\mathcal{F}_i/\mathcal{F}_{i-1} \simeq \mathcal{L}(-\delta_i)$, where $\delta_i$ is the character of $T_D$ taking $\diag (t_1, \dots, t_n, t_n^{-1}, \dots, t_1^{-1}) \mapsto t_i$,
\begin{equation} \label{neweqn2}
x_i=-c_1( \mathcal{F}_i/\mathcal{F}_{i-1})
= \beta(\delta_i).
\end{equation}
Now, by \cite[Planche IV]{Bo}, combining the equations \eqref{neweqn1} and \eqref{neweqn2}, we get
\begin{align*} \epsilon_{s_j}&=\beta(\delta_1+\cdots + \delta_j)= x_1 +\dots +x_j,\,\,\,\text{for $1\leq j\leq n-2$}\\
  &= \beta\left(\frac{1}{2}(\delta_1+\cdots + \delta_{n-1}-\delta_n)\right)=  
  \frac{1}{2}(x_1+\cdots + x_{n-1}-x_n),\,\,\,\text{for $j= n-1$}\\  
  &=  \beta\left(\frac{1}{2}(\delta_1+\cdots + \delta_n)\right)=   
  \frac{1}{2}(x_1+\cdots + x_n),\,\,\,\text{for $ j=n$}.
  \end{align*} 
This proves the first part of the lemma. 

Further, by \cite[Proposition 24(c)]{Ku2}, 
\begin{align*} &\xi^B(\bar{t}_j)\\
&=\epsilon_{s_j}- \epsilon_{s_{j-1}}= x_j,\,\,\,\text{for $1\leq j\leq n-2$}\\
  &= \epsilon_{s_{n-1}}+\epsilon_{s_n}- \epsilon_{s_{n-2}} =\\
  &
   \frac{1}{2}(x_1+\cdots + x_{n-1}-x_n)+  \frac{1}{2}(x_1+\cdots + x_{n-1}+x_n)  
    -(x_1+\cdots + x_{n-2})  =x_{n-1}  
   ,\,\text{for $j= n-1$}\\  
  &=  \epsilon_{s_{n}}-\epsilon_{s_{n-1}}
 =   
  \frac{1}{2}(x_1+\cdots + x_n)-\frac{1}{2}(x_1+ \dots +x_{n-1}-x_n) = x_n,\,\,\,\text{for $ j=n$}.
  \end{align*} 
  This proves the second part of the lemma.
  \end{proof}
  
By Lemma \ref{epsilon s_j}, we get:
\begin{lemma}\label{CQ and CQ*}
 For $0\le j\le n$, $$c(\mc[Q]_j)c(\mc[Q]_j^*)=\prod_{p=j+1}^n(1-x_p)(1+x_p) \in H^*(\Fl, \Z) ,$$
where $\mc[Q]_j :=\mc[E]/\mc[F]_j$ and $c(\mc[Q]_j)$ is the total Chern class of $\mc[Q]_j$. 
\end{lemma}

\begin{proof}
    By the definition,  we have $1-x_p=c(\mc[F]_p/\mc[F]_{p-1})$. Then, 
    $$\prod_{p=j+1}^n(1-x_p)(1+x_p)=\prod_{p=j+1}^n c(\mc[F]_p/\mc[F]_{p-1})c((\mc[F]_p/\mc[F]_{p-1})^*)=\frac{c(\mc[F]_n)}{c(\mc[F]_j)}\frac{c(\mc[F]_n^*)}{c(\mc[F]_j^*)}.$$
    From the exact sequence $0\to \mc[F]_j\to \mc[E]\to \mc[Q]_j\to 0$, we get 
    $$c(\mc[Q]_j)c(\mc[F]_j)=1\text{ and }c(\mc[Q]_j^*)c(\mc[F]_j^*)=1.$$
    In particular, we have $c(\mc[Q]_n)c(\mc[F]_n)=1$. From the bilinear form we have $\mc[Q]_n=(\mc[F]_n^\perp)^*$. Further,  from the definition,  we have $\mc[F]_n^\perp=\mc[F]_n$. Thus, we have 
    $$c(\mc[F]_n)c(\mc[F]_n^*)=1.$$
    This proves the lemma.
\end{proof}

The above lemma allows us to prove the following crucial result of the paper in the case of $\SO(2n)$.

\begin{thm}\label{xi image D}
For $2\le k\le n-1$, the map $
\xi^{P_{n-k}}:\rep^\C_{\poly}(L_{n-k})\to H^*(\OG(n-k,2n),\C)$ takes 
$$e_i((\overline{t}_{n-k+1})^2,\ldots,(\overline{t}_{n})^2)\mapsto c_i^2+2\sum_{j=1}^i(-1)^jc_{i+j}c_{i-j},\,\,\text{for any } 1\le i\le k,$$
where $c_p=c_p(\mc[Q])$ is the $p$-th Chern class of the quotient bundle $\mc[Q]$ as in \eqref{ses_grmn} and $e_i$ is the $i$-th elementary symmetric polynomial. 

\end{thm}
\begin{proof}
This follows from Lemmas \ref{epsilon s_j} and 
\ref{CQ and CQ*} by taking the degree $2i$ component in Lemma \ref{CQ and CQ*} for $j=n-k$,  and using the projection $\Fl\to \OG(n-k,2n)$.
\end{proof}

For any $k\geq 2$, consider the inverse system: 
\begin{equation}\label{inverse system D}
    \cdots\leftarrow H^*(\OG(n-k,2n),\Z)\xleftarrow{\pi_n^*} H^*(\OG(n+1-k,2(n+1)),\Z)\leftarrow\cdots,
\end{equation}
where $\pi_n:\OG(n-k,2n)\hookrightarrow \OG(n+1-k,2(n+1))$ is the embedding $S\to T_n(S)\oplus \C e_{n}$. Here, $T_n:\C^{2n}\to \C^{2n+2}$ is the linear embedding taking $e_i\mapsto e_i$ for $1\le i\le n-1$, taking $e_i\mapsto e_{i+1}$ for $n\le i\le n+1$, and $e_i\mapsto e_{i+2}$ for $n+2\le i\le 2n$. Notice that in the definition of the isotropic flag $F_\bullet$ we fix an isotropic subspace $L_n \subset \C^{2n}$ for each $n \geq 4$. To define a compatible isotropic subspace $L_{n+1}\subset \C^{2n+2}$ (under the embedding $T_n$), we set $L_{n+1} = T_n(L_n)\oplus \C e_n$.\\

 Define the stable cohomology ring as $$\mathbb{H}^*(\OG_k,\Z)=\varprojlim_n H^*(\OG(n-k,2n),\Z).$$ 
 
 Then, we have the following proposition.
 \begin{prop}\label{stable cohomology presentation} For any $k\geq 2$, the stable cohomology ring  $\mathbb{H}^*(\OG_k,\Z) $
 admits a presentation as polynomial ring $\Z[\tau_1,\ldots,\tau_{k-1},\tau_k,\tau_k',\tau_{k+1},\ldots]$ modulo the relations: 
$$\tau_s^2+\sum_{p=1}^s(-1)^p\tau_{s+p}c_{s-p}=0,\ s\ge k+1,$$
$$\tau_k\tau_k'+\sum_{p=1}^k(-1)^p\tau_{k+p}\tau_{k-p}=0,$$
where we interpret $\tau_0=1$ and 
\begin{equation*}
    c_p=
    \begin{cases}
      \tau_p, & \text{if }p<k, \\
      \tau_k+\tau'_k, & \text{if }p=k, \\
      2\tau_p, & \text{if }p>k.
    \end{cases}
  \end{equation*}
Further,  the natural restriction homomorphism $\varphi_{k,n}:\mathbb{H}^*(\OG_k,\Z)\to H^*(\OG(n-k,2n),\Z)$ takes $c_p$ to $c_p(\mc[Q])$. Moreover, it is surjective. 
\end{prop}
\begin{proof}
 Recall from \eqref{c_p describe} that for $c_p(n)=c_p(\mc[Q])$, we have 
 \begin{equation*}
    c_p(n)=
    \begin{cases}
      \tau_p, & \text{if }p<k, \\
      \tau_k+\tau'_k, & \text{if }p=k, \\
      2\tau_p, & \text{if }p>k,
    \end{cases}
  \end{equation*}
  where $\mc[Q]$ is the tautological quotient bundle over $\OG(n-k,2n)$.  (Here we use the notation $c_p(n)$ instead of $c_p$ since we will need to vary $n$.)  
  
  From the functoriality of Chern classes, $\pi_n^*$ takes $c_p(n+1)\to c_p(n)$ for $1\le p\le n+k$, and $c_{n+k+1}(n+1)\mapsto 0$. Also,   $\pi_n^*$ takes $\tau_k$ to $\tau_k$, and $\tau_k'$ to $\tau_k'$:
  
  To prove this, represent $\tau_k=\tau_k^n\in H^{2k}(\OG(n-k,2n), \Z)$ as the Schubert class $\epsilon^{P_{n-k}}_{w_n}$, where $w_n:=s_{n-1}\dots s_{n-k}$. Then, $w_n\cdot (\C e_1 \oplus \dots \oplus \C e_{n-k})=   \C e_1 \oplus \dots \oplus \C e_{n-k-1} \oplus \C e_n$, under the standard basis $\{e_i\}_{1\leq i\leq 2n}$ of $\C^{2n}$. Moreover,
  $$\pi_n\left((\C e_1 \oplus \dots \oplus \C e_{n-k-1} \oplus \C e_n)\right)=   \C e_1 \oplus \dots \oplus \C e_{n-k-1} \oplus \C e_{n+1}\oplus \C e_n.$$
  Consider the (ordered) basis $\tilde{\bf e}:= \{e_1, \dots, e_{n-k-1}, e_n, e_{n-k}, \dots, e_{n-1},\hat{e}_n, e_{n+1}, \dots, e_{2n+2}\}$ of $\C^{2n+2}$. Then,
  $$w_{n+1}\cdot (\C e_1 \oplus \dots \oplus \C e_{n-k-1}\oplus \C e_n\oplus \C e_{n-k})=   \C e_1 \oplus \dots \oplus \C e_{n-k-1} \oplus \C e_n\oplus \C e_{n+1}.$$  
  Thus,   $$\pi_n\left(w_n\cdot (\C e_1 \oplus \dots \oplus \C e_{n-k})\right)= w_{n+1}\cdot  (\C e_1 \oplus \dots \oplus \C e_{n-k-1} \oplus \C e_{n-k}\oplus \C e_{n}).$$
  Further, it is easy to see that
  \begin{equation} \label{neweqn105}  
  \pi_n\left(B_n^-\cdot w_n \cdot  (\C e_1 \oplus \dots \oplus \C  e_{n-k})\right) \subset \tilde{B}_{n+1}^- \cdot \pi_n \left(
  w_n \cdot  (\C e_1 \oplus \dots \oplus \C e_{n-k})\right),
  \end{equation}
  where $B^-_n$ is the Borel subgroup opposite to the standard Borel subgroup $B_n$ of $\SO(2n)$ in the $\{e_i\}_{1\leq i \leq 2n}$ basis of $\C^{2n}$ and $\tilde{B}_{n+1}^-$ is the   Borel subgroup opposite to the standard Borel subgroup  of $\SO(2n+2)$, but under  the $\tilde{\bf e}$ basis of $\C^{2n+2}$. Now, since $\pi_n$ is an embedding and the two sides of the equation \eqref{neweqn105} have the same dimension ($=k$), the inclusion in   the equation \eqref{neweqn105}  is an equality. This proves that $\tau_k^{n+1}$ restricts to $\tau_k^n$ under $\pi_n^*$. Since $\tau_k+\tau_k' =c_k$ and $c_k$ restricts to $c_k$, we get that $\tau_k'$ also restricts to $\tau'_k$. 
  
 Thus, from the presentation  of the ring $H^*(\OG(n-k,2n),\Z)$ (cf. Theorem \ref{presentation cohomology D}),  
   the first three relations \eqref{presentation T1}--\eqref{presentation T3} disappear  in the inverse limit. This proves the presentation of $\mathbb{H}^*(\OG_k,\Z)$ as in the proposition.
   
   For each $\lambda\in\Tilde{\mc[P]}(k,n)$ we define a monomial $\tau^\lambda$ in terms of the special Schubert classes as follows. If $\text{type}(\lambda)\ne 2$, then set $\tau^\lambda=\prod_i\tau_{\lambda_i}$. If $\text{type}(\lambda)=2$ then $\tau^\lambda$ is defined by the same product formula, but replacing each occurrence of $\tau_k$ with $\tau_k'$.   Then, $\{\tau^\lambda\}_{\lambda \in   \Tilde{\mc[P]}(k,n)}$ form a $\Z$-basis of $H^*(\OG(r, 2n), \Z)$ (cf. \cite[Theorem 3.2(b)]{BKT1}).
      
   From the description of $\pi_n^*$,  the natural ring homomorphism $\varphi_{k,n}:\mathbb{H}^*(\OG_k,\Z)\to H^*(\OG(n-k,2n),\Z)$ takes $\tau^\lambda$ to $\tau^\lambda$ whenever $\lambda \in  \Tilde{\mc[P]}(k,n)$   
   and zero otherwise. In particular, $\varphi_{k,n}$ is surjective since $\{\tau^\lambda\}_{\lambda \in   \Tilde{\mc[P]}(k,n)}$ span  $H^*(\OG(r, 2n), \Z)$.  
\end{proof}

By \cite[Proposition 12]{Ku2},  
\begin{equation} \label{neweqn3} \rep^\C_{\poly}(\SO(2k))\simeq \C_{\sym}[\bar{h}_1^2,\ldots,\bar{h}_k^2]\oplus \left(\bar{h}_{1,k}\C_{\sym}[\bar{h}_1^2,\ldots,\bar{h}_k^2]\right),
\end{equation} 
 where $\bar{h}_{1,k} :=\bar{h}_1\bar{h}_2\cdots \bar{h}_k$.  Define (using the decomposition \eqref{levi D}) 
$$\iota_k^n:\rep^\C_{\poly}(\SO(2k))\to \rep^\C_{\poly}(L^D_{n-k})$$
by taking $f(\bar{\bf{h}})\to 1\otimes f(\bar{\bf{t}})$, where $\bar{\bf{h}} := (\bar{h}_1, \dots, \bar{h}_k), \bar{\bf{t}} :=
(\bar{t}_{n-k+1}, \dots, \bar{t}_n)$ and $f(\bar{\bf{t}})$ is the same polynomial written in the $\bar{\bf{t}}$ variables under the transformation  $\bar{t}_{n-k+i}=\bar{h}_i$. This gives rise to a $\mathbb{C}$-algebra homomorphism  $$\xi_{n,k}=\xi^{P_{n-k}}\circ \iota_k^n:\rep^\C_{\poly}(\SO(2k))\to H^*(\OG(n-k,2n),\C).$$ Moreover, by virtue of Theorem
\ref{xi image D} and the isomorphism \eqref{neweqn3},
 $(\xi_{n,k})_{| \C_{\sym}[\bar{h}_1^2,\ldots,\bar{h}_k^2]} 
 $ commutes with the following inverse system:
\begin{center}
\begin{tikzcd}
     & \phantom{H^*(\OG(n-k,2n),\C)} \\
    \rep^\C_{\poly}(\SO(2k)) \arrow{r}{\xi_{n,k}} \arrow{dr}{\xi_{n+1,k}}
    & H^*(\OG(n-k,2n),\C) \arrow{u}{\pi_{n-1}^*}\\
     &H^*(\OG(n-k+1,2n+2),\C) \arrow{u}{\pi_{n}^*} \\
     & \phantom{H^*(\OG(n-k,2n),\C)} \arrow{u}{\pi_{n+1}^*}
\end{tikzcd}
\end{center}
We next show that 
\begin{equation} \label{neweqn5} 
\pi_n^*\circ \xi_{n+1, k} ( \bar{h}_{1,k}) =  \xi_{n, k} ( \bar{h}_{1,k}).
\end{equation}
To prove this, define a morphism $\hat{\pi}_n:\Fl_n \to \Fl_{n+1}$ by
\begin{align*} &\hat{\pi}_n \left((0=F_0\subset F_1\subset \dots \subset F_{2n}=\C^{2n})\right)\\
&= (0\subset T_n(F_1)\subset \dots \subset T_n(F_{n-k-1})\subset \widehat{T_n(F_{n-k})}\subset T_n(F_{n-k})\oplus \C e_n\\
&\subset \dots \subset T_n(F_n)\oplus \C e_n\subset (T_n(F_{n-1})\oplus \C e_n)^\perp\subset \dots\subset (T_n(F_{n-k})\oplus \C e_n)^\perp \subset \\
&\widehat{(T_n(F_{n-k}))^\perp}\subset T_n(F_{n-k-1})^\perp\subset \dots \subset T_n(F_1)^\perp \subset \C^{2n+2}),
\end{align*}
where we map $F_i$ successively to the right side omitting the two terms with $\widehat{}$ over them. 
Clearly, $T_n(F_n)\oplus \C e_n$ and $L_{n+1}$ are in the same family. 
We have the following commutative diagram:
\[
\xymatrix{
\Fl_n  \ar[d]\ar[r]^{\hat{\pi}_n} & \Fl_{n+1}
\ar[d]\\
OG(n-k, 2n)\ar[r]^{\pi_n} & OG(n+1-k, 2(n+1)),
}
\]
where the vertical maps are the canonical projections. From the above definition of $\hat{\pi}_n$, it is clear that the line bundle $\mathcal{F}_{j+1}/\mathcal{F}_j$ over $\Fl_{n+1}$ pulls back under $\hat{\pi}_n$ to the line bundle  $\mathcal{F}_{j}/\mathcal{F}_{j-1}$ over $\Fl_n$ for any $n-k+1\leq j\leq n$. Thus, by Lemma \ref{epsilon s_j} and the injectivity of 
$H^*(OG(n-k, 2n), \Z) \to H^*(\Fl_n, \Z)$, we get the validity of identity \eqref{neweqn5}.

Thus, we get a  $\mathbb{C}$-algebra homomorphism 
$$\xi_k:\rep^\C_{\poly}(\SO(2k))\to \mathbb{H}^*(\OG_k,\C).$$

The following theorem is one of our main results for $D_n$-type groups.
\begin{thm}\label{injectivity}
Let $k\geq 2$. Then, the above $\mathbb{C}$-algebra homomorphism $\xi_k:\rep^\C_{\poly}(\SO(2k))\to \mathbb{H}^*(\OG_k,\C)$ takes the generator
\begin{equation}\label{xi k gen}
    e_i((\bar{h}_1)^2,\ldots,(\bar{h}_k)^2)\mapsto c_i^2+2\sum_{j=1}^i(-1)^jc_{i+j}c_{i-j},\,\,\text{for any } 1\le i\le k,
\end{equation} 
where $c_i$ are as in Proposition \ref{stable cohomology presentation}.

Moreover,  $\xi_k$ is injective.
\end{thm}

\begin{proof}
We will simply abbreviate  $e_i((\bar{h}_1)^2,\ldots,(\bar{h}_k)^2)$ by $e_i$.
The first part of the theorem (i.e., \eqref{xi k gen}) follows directly from Theorem \ref{xi image D} and Proposition \ref{stable cohomology presentation}.

For the injectivity of $\xi_k$, since $\C$ is a torsionfree $\Z$-module, we only need to prove that 
$$\xi_k^\Z:\rep^\Z_{\poly}(\SO(2k))\to \mathbb{H}^*(\OG_k,\Z)$$
is injective (cf. \cite[Chap. 5, Sec. 2, Lemma 5]{Sp}), where 
$$\rep^\Z_{\poly}(\SO(2k)):= \Z_{\sym} [\bar{h}_1^2,\ldots,\bar{h}_k^2]\oplus \left(\bar{h}_{1,k}\Z_{\sym}[\bar{h}_1^2,\ldots,\bar{h}_k^2]\right).$$
First, we take the subring $\bar{\mathbb{H}}^*(\OG_k,\Z)$ of $\mathbb{H}^*(\OG_k,\Z)$ generated by $\{\tau_k-\tau_k',c_1,c_2,\ldots,c_k,\ldots\}$. Consider the restricted homomorphism: $$\eta_k^\Z:\Z_{\sym}[\bar{h}_1^2,\ldots,\bar{h}_k^2]\to \bar{\mathbb{H}}^*(\OG_k,\Z),$$
which is well defined by \eqref{xi k gen}. Observe that $\eta_k^\Z$ is a homomorphism of graded rings if we assign degree $4i$ to each $e_i$ and the standard cohomological degree to  $\bar{\mathbb{H}}^*(\OG_k,\Z)$. 

We now prove that $\eta_k^\Z$ is injective:

From Proposition \ref{stable cohomology presentation},  $\bar{\mathbb{H}}^*(\OG_k,\Z)$ has the presentation as polynomial ring $\Z[\tau_k-\tau_k',c_1,c_2,\ldots,c_k,\ldots]$ modulo the relations: 
$$c_s^2+2\sum_{p=1}^s(-1)^p c_{s+p}c_{s-p}=0,\ s\ge k+1,$$
$$c_k^2-(\tau_k-\tau_k')^2+2\sum_{p=1}^k(-1)^pc_{k+p}c_{k-p}=0.$$
Since $\bar{\mathbb{H}}^*(\OG_k,\Z)$ is free $\Z$-module of finite rank in each degree (since so is $ \mathbb{H}^*(\OG_k,\Z)$), we have exact sequence (cf. \cite[Chap. 5, Sec. 2, Lemma 5]{Sp})

$$0\to \Z_2\otimes_\Z \bar{K} \to\Z_2\otimes_\Z \Z_{\sym}[\bar{h}_1^2,\ldots,\bar{h}_k^2],$$
where $\Z_2 := \Z/(2)$ and $\bar{K}$ is the kernel of $\eta_k^\Z$. We next prove that 
$$\eta_k^{\Z_2}:\Z_2\otimes_\Z \Z_{\sym}[\bar{h}_1^2,\ldots,\bar{h}_k^2] \to\Z_2\otimes_\Z \bar{\mathbb{H}}^*(\OG_k,\Z)$$
is injective:

To prove this, observe that by the above presentation of  $\bar{\mathbb{H}}^*(\OG_k,\Z)$,
We have 
$$\Z_2\otimes_\Z \bar{\mathbb{H}}^*(\OG_k,\Z)\simeq\Z_2[c_1,\ldots,c_{k-1}]\otimes \frac{\Z_2[\tau_k-\tau_k',c_k]}{\langle(\tau_k-\tau_k')^2-c_k^2\rangle}\otimes \frac{\Z_2[c_{k+1},c_{k+2},\ldots]}{\langle c_{k+1}^2,c_{k+2}^2,\ldots\rangle}.$$
Moreover,  by the equation \eqref{xi k gen},
the ring homomorphism $\eta_k^{\Z_2}$ takes:

\begin{equation}\label{eta_z2_relation1}
    e_i\mapsto c_i^2,\text{ for any }1\le i\le k-1,
\end{equation}
\begin{equation}\label{eta_z2_relation2}
    e_k\mapsto (\tau_k-\tau_k')^2=c_k^2.
\end{equation}
Thus,  we  conclude that $\eta_k^{\Z_2}$ is injective, so $\Z_2\otimes_\Z \bar{K}=0$. Since $\bar{K}$ is free $\Z$-module of finite rank in each degree, we have $\bar{K}=0$, so $\eta_k^{\Z}$ is injective.

We are now ready to prove the injectivity of $\xi_k^\Z$ (and hence that of $\xi_k$):

Any element in $\rep^\Z_{\poly}(\SO(2k))$ can be written as $f+\bar{h}_{1,k}g$, where $f,g\in\Z[e_1,\ldots,e_k]$. Now, assume that 
\begin{equation}\label{neweqn4}\xi_k^\Z(f+\bar{h}_{1,k}g)=0.
\end{equation}
In particular, 
\begin{equation}\label{xi_square_equal}
 \eta_k^\Z(f^2)=   \xi_k^\Z(f^2)=\xi_k^\Z(\bar{h}^2_{1,k}g^2) =\eta_k^\Z(\bar{h}_{1, k}^2g^2).
\end{equation}

Notice that we have  $\bar{h}_{1,k}^2=e_k((\bar{h}_1)^2,\ldots,(\bar{h}_k)^2)$. 
Thus, we have 
$$\eta_k^{\Z}(f^2(e_1,\ldots,e_k))=\eta_k^{\Z}(e_k\cdot g^2(e_1,\ldots,e_k)).$$
If $g\ne 0$, let the highest degree of $e_k$ in $f,g$ be $p,q$ respectively. Then, the highest degree of $e_k$ in  $f^2(e_1,\ldots, e_k)$ and $e_kg^2(e_1,\ldots, e_k)$ are $2p$ and $2q+1$ respectively. In particular, $f^2\neq e_kg^2$ and hence $  \eta_k^\Z(f^2)\neq \eta_k^\Z(\bar{h}_{1, k}^2g^2)$ (due to the injectivity of $\eta_k^\Z$). This contradicts
\eqref{xi_square_equal} and hence $g=0$ (and hence so is $f=0$) by \eqref{neweqn4}. This proves the injectivity of 
 $\xi_k^\Z$ completing the proof of the theorem.
\end{proof}

\begin{remark}\label{newremark}
The homomorphism $\xi_k:\rep^\C_{\poly}(\SO(2k))\to \mathbb{H}^*(\OG_k,\C)$ 
as in Theorem \ref{injectivity} is {\it not} surjective since the domain is finitely generated whereas the range is not.
\end{remark}

\section{Exceptional Groups}
In the discussion of exceptional groups in the following Sections 4-8, we identify the maximal torus $T$ of connected and simply-connected $G$ by
\begin{equation} \label{neweqn6}
T= \Hom_\Z(\ft^*_\Z, \C^*),
\end{equation}
where $\ft_\Z := \oplus_{i=1}^\ell \Z \alpha_i^\vee$ ($\alpha_i^\vee$ being the simple coroots) and  $\ft_\Z^* := \Hom_\Z(\ft_\Z, \Z)$ is the weight lattice. 

In the following sections, we use the coordinates $(\bar{t}_1, \dots, \bar{t}_\ell)$ of any element $\bar{t}\in \ft$ in the basis $\{\alpha_1^\vee, \dots, \alpha_\ell^\vee\}$.

We use the following in the calculations in the following sections.

\vskip1ex

$\bullet$ The Springer morphism $\theta_\lambda$ restricted to $T$ is determined by using \cite[Theorems 1, 2]{R}. The actual calculation can be found in the link \cite{X}.

\vskip1ex
$\bullet$ Let $L_r$  be the Levi subgroup  (containing the maximal torus) of a maximal parabolic subgroup $P_r$ of $G$ and let $W_r$ be its Weyl group. 
We determine the invariant ring $S(\ft^*)^{W_r}$ by first considering the action of the generators of $W_r$ on $\ft^*$ in the basis consisting of the simple roots making use of \cite[Proposition 12]{Ku2} when $L_r$ is of classical type. When $L_r$ is of exceptional type, we make use of 
 \cite[$\S$3]{lee}. Then, we change the simple root basis to the basis consisting of the fundamental weights by using the Cartan matrix.

\vskip1ex

$\bullet$ We explicitly determine the cup product in $H^*(G/B, \Z)$ by using the following theorem of Duan. The actual calculation can be found in the link \cite{X}.

Write, for any $u,v\in W$,
$$ \epsilon^B_u\cdot \epsilon^B_v=\sum_{w\in W}\, c^w_{u,v} \epsilon^B_w.$$ 

For any strictly upper triangular square matrix $A$ of size $k\times k$, Duan defines an operator (cf. \cite[Definition 2]{D})
$$T_A: \Z[x_1, \dots, x_k]^{(k)}\to \Z,$$
where the superscript $(k)$ denotes the subspace of the homogeneous polynomials of degree $k$ (assigning degree $1$ to each $x_i$). Now, fix a reduced decomposition $s_{i_1}\dots s_{i_k}$ of $w$ in terms of simple reflections
(of length $k$) and define the $k\times k$-matrix $A_w = (a_{p,q})$ by $a_{p, q} :=0 $ if $p\geq q$ and $a_{p,q} := -\alpha_{i_p}(\alpha_{i_q}^\vee)$ for $p<q$. 

 Then, by \cite[Theorem 2.3]{D}, 
\begin{thm}\label{duan2005multiplicative}
Take $u,v,w\in W$ with $k=\ell(w)=\ell(u)+\ell(v)$ and  fix a reduced decomposition $s_{i_1}\dots s_{i_k}$ of $w$. Then,  
 $$c^w_{u,v} = T_{A_w}\left(\left( \sum_{|L|=\ell(u),s_L=u}x_L\right) \left( \sum_{|K|=\ell(v),s_K=v}x_K\right) \right),$$
 where  $L,K\subset \{1, \dots, k\}$, $s_L := \prod_{j\in L}\,s_{i_j}$ (product in the ascending order of $j$), and $x_L := \prod_{j\in L} x_j$.  
 
The result does not depend on the choice of the reduced decomposition of $w$.
\end{thm}

In the following sections, we only consider the standard maximal parabolic subgroups $P_r $ ($1\leq r \leq \ell$), where its Levi subgroup $L_r$ containing $T$ has for its simple roots all the simple roots except $\alpha_r$. 
 
For our Springer morphism, we take the weight $\lambda$ with the minimum Dynkin index. Then, $\lambda =\omega_1$ (for $G_2$);   $\lambda =\omega_4$ (for $F_4$);   $\lambda =\omega_1$ (for $E_6$);  $\lambda =\omega_7$ (for $E_7$);  $\lambda =\omega_8$ (for $E_8$) (cf. \cite[Corollary A.9]{Ku3}).  With this choice, we abbreviate $\Rep^\C_{\lambda-\poly}(L_r)$ by $\Rep^\C_{\poly}(L_r)$ and $\xi^{P_r}_{\lambda}$ by 
 $\xi^{P_r}$. 
 
 \vskip1ex
 
 {\it We follow the indexing convention as in \cite[Planche V-IX]{Bo}.} 
 \section{$G_2$ }
 
 We take the basis $x_1=\omega_1, x_2= \omega_2-\omega_1$ of $\ft^*_\Z$ and use the coordinates $(t_1, t_2)$ of $t\in T$ defined by (under the identification \eqref{neweqn6}) $t_i=t(x_i)$.  Then, the Springer morphism $\theta_{\omega_1}$ is given as follows:
 \begin{lemma} \label{lem3.1}
 \begin{equation}
    \theta_{\omega_1}(t_1,t_2)= \left(\Theta_1(t)=\frac{1}{6}(2t_1-2t_1^{-1} +t_2 -t_2^{-1}+t_1t_2^{-1}-t_1^{-1}t_2),\Theta_2(t)= \frac{1}{2}(t_1-t_1^{-1}-t_2^{-1}+t_2)\right).
   \end{equation} 
 \end{lemma}   
    
    It is easy to see that 
    $$s_1\cdot x_1=x_2,\ s_1\cdot x_2=x_1,\ s_2\cdot x_1=x_1,\ s_2\cdot x_2=x_1-x_2.$$
    Thus, $s_1$ switch $x_1$ and $x_2$, and $s_2$ fixes $x_1$ and $(x_2-\frac{1}{2}x_1)^2$ and hence
      $$S(\mf[t]^*)^{W_1}=\C [x_1, (x_2-\frac{1}{2}x_1)^2]\,\,\text{and}\,\,
S(\mf[t]^*)^{W_2}=\C_{\sym}\left[x_1, x_2\right],$$  
where $W_i$ is the Weyl group of $L_i$.  

Combining the equation \eqref{eqnborel} with Lemma \ref{lem3.1} and using Theorem \ref{duan2005multiplicative}, we get the following:
\begin{theorem}  The generators of $\Rep^\C_{\poly}(L_1)$ under $\xi^{P_1}$ are mapped as follows:
$$\Theta_1(t)
\mapsto \epsilon^{P_1}_{s_1},\ (\Theta_2(t)-\frac{3}{2}\Theta_1(t))^2\mapsto \frac{3}{4}\epsilon_{s_2s_1}^{P_1}.$$
Similarly, under  $\xi^{P_2}$ we get
$$\Theta_2(t)\mapsto \epsilon^{P_2}_{s_2},\ \Theta_1(t)\Theta_2(t)-\Theta_1^2(t)\mapsto \epsilon_{s_1s_2}^{P_2}.$$
\end{theorem}

\section{$F_4$}
We use $\{\omega_1,\omega_2,\omega_3,\omega_4\}$ as a basis for $\mf[t]^*$. This gives rise to a coordinate system on 
$T$ as earlier using the identity \eqref{neweqn6}.
\begin{lemma} \label{lem4.1} \begin{equation}
    \theta_{\omega_4}(t_1,t_2, t_3, t_4)= \left(\Theta_1,\Theta_2, \Theta_3, \Theta_4)\right),
   \end{equation} 
   where 
   \begin{align*}
   \Theta_1(t)= & \frac{1}{12}\big( 2t_4^{1}+2t_3^{1}t_4^{-1}+2t_2^{1}t_3^{-1}+2t_1^{1}t_2^{-1}t_3^{1}+2t_1^{1}t_3^{-1}t_4^{1}+2t_1^{1}t_4^{-1}\\
&-2t_1^{-1}t_4^{1}-2t_1^{-1}t_3^{1}t_4^{-1}-2t_1^{-1}t_2^{1}t_3^{-1}-2t_2^{-1}t_3^{1}-2t_3^{-1}t_4^{1}-2t_4^{-1} \big).
  \end{align*}
\begin{align*}
\Theta_2(t)= & \frac{1}{12}\big( 4t_4^{1}+4t_3^{1}t_4^{-1}+4t_2^{1}t_3^{-1}+2t_1^{1}t_2^{-1}t_3^{1}+2t_1^{-1}t_3^{1}+2t_1^{1}t_3^{-1}t_4^{1}\\
&+2t_1^{-1}t_2^{1}t_3^{-1}t_4^{1}+2t_1^{1}t_4^{-1}+2t_1^{-1}t_2^{1}t_4^{-1}-2t_1^{1}t_2^{-1}t_4^{1}-2t_1^{1}t_2^{-1}t_3^{1}t_4^{-1}-2t_1^{-1}t_4^{1}\\
&-2t_1^{-1}t_3^{1}t_4^{-1}-2t_1^{1}t_3^{-1}-2t_1^{-1}t_2^{1}t_3^{-1}-4t_2^{-1}t_3^{1}-4t_3^{-1}t_4^{1}-4t_4^{-1} \big).   
\end{align*}
\begin{align*}
    \Theta_3(t)= & \frac{1}{12}\big( 3t_4^{1}+3t_3^{1}t_4^{-1}+2t_2^{1}t_3^{-1}+2t_1^{1}t_2^{-1}t_3^{1}+2t_1^{-1}t_3^{1}+t_1^{1}t_3^{-1}t_4^{1}\\
&+t_1^{-1}t_2^{1}t_3^{-1}t_4^{1}+t_1^{1}t_4^{-1}+t_2^{-1}t_3^{1}t_4^{1}+t_1^{-1}t_2^{1}t_4^{-1}+t_2^{-1}t_3^{2}t_4^{-1}-t_2^{1}t_3^{-2}t_4^{1}\\
&-t_2^{1}t_3^{-1}t_4^{-1}-t_1^{1}t_2^{-1}t_4^{1}-t_1^{1}t_2^{-1}t_3^{1}t_4^{-1}-t_1^{-1}t_4^{1}-t_1^{-1}t_3^{1}t_4^{-1}-2t_1^{1}t_3^{-1}\\
&-2t_1^{-1}t_2^{1}t_3^{-1}-2t_2^{-1}t_3^{1}-3t_3^{-1}t_4^{1}-3t_4^{-1} \big).   
   \end{align*}
\begin{align*}
\Theta_4(t)= & \frac{1}{12}\big( 2t_4^{1}+t_3^{1}t_4^{-1}+t_2^{1}t_3^{-1}+t_1^{1}t_2^{-1}t_3^{1}+t_1^{-1}t_3^{1}+t_1^{1}t_3^{-1}t_4^{1}\\
&+t_1^{-1}t_2^{1}t_3^{-1}t_4^{1}+t_2^{-1}t_3^{1}t_4^{1}+t_3^{-1}t_4^{2}-t_3^{1}t_4^{-2}-t_2^{1}t_3^{-1}t_4^{-1}-t_1^{1}t_2^{-1}t_3^{1}t_4^{-1}\\
&-t_1^{-1}t_3^{1}t_4^{-1}-t_1^{1}t_3^{-1}-t_1^{-1}t_2^{1}t_3^{-1}-t_2^{-1}t_3^{1}-t_3^{-1}t_4^{1}-2t_4^{-1} \big).
\end{align*}
    \end{lemma}
    Combining the equation \eqref{eqnborel} with Lemma \ref{lem4.1} and using Theorem \ref{duan2005multiplicative}, we get the following:
\begin{theorem}  The generators of $\Rep^\C_{\poly}(L_r)$ under $\xi^{P_r}$ are mapped as follows:

(a) ($L_1$):  Let 
$$y_1^{L_1}=-\Theta_1+2\Theta_4,\ y_2^{L_1}=-\Theta_1+2\Theta_3-2\Theta_4,\ y_3^{L_1}=-\Theta_1+2\Theta_2-2\Theta_3,\ y_4^{L_1}=\Theta_1.$$
Then, 
$$\Rep^\C_{\poly}(L_1)
=\C_{\sym}\left[(y_1^{L_1})^2, (y_2^{L_1})^2, (y_3^{L_1})^2 \right] \otimes_\C \C[y_4^{L_1}].$$
Further,  the generators of $\Rep^\C_{\poly}(L_1)$ 
under $\xi^{P_1}$ go to:
$$e_1\left((y_1^{L_1})^2,(y_2^{L_1})^2,(y_3^{L_1})^2\right)\mapsto -\epsilon_{s_2s_1}^{P_1},$$
$$e_2\left((y_1^{L_1})^2,(y_2^{L_1})^2,(y_3^{L_1})^2\right)\mapsto -10\epsilon_{s_2s_3s_2s_1}^{P_1}+28\epsilon_{s_4s_3s_2s_1}^{P_1},$$
$$e_3\left((y_1^{L_1})^2,(y_2^{L_1})^2,(y_3^{L_1})^2\right)\mapsto 12\epsilon_{s_4s_1s_2s_3s_2s_1}^{P_1}-16\epsilon_{s_3s_4s_2s_3s_2s_1}^{P_1},$$
$$y_4^{L_1}\mapsto \epsilon_{s_1}^{P_1}.$$
(b)  ($L_2$):  Let 
$$y_1^{L_2}=-3\Theta_1+2\Theta_2,\ y_2^{L_2}=3\Theta_1-\Theta_2,\ y_3^{L_2}=3\Theta_2-4\Theta_3,\ y_4^{L_2}=-\Theta_2+4\Theta_3-4\Theta_4,\ y_5^{L_2}=-\Theta_2+4\Theta_4.$$
Then, 
$$\Rep^\C_{\poly}(L_2)
=\C_{\sym}[y_1^{L_2},y_2^{L_2}] \otimes_\C  \C_{\sym}\left[y_3^{L_2}, y_4^{L_2}, y_5^{L_2} \right]\Big/\left(y_1^{L_2}+y_2^{L_2}-y_3^{L_2}-y_4^{L_2}-y_5^{L_2}\right).$$
Further, the generators of $\Rep^\C_{\poly}(L_2)$ 
under $\xi^{P_2}$ go to:
$$e_1(y_1^{L_2},y_2^{L_2})\mapsto \epsilon_{s_2}^{P_2},$$
$$e_2(y_1^{L_2},y_2^{L_2})\mapsto 7\epsilon_{s_1s_2}^{P_2}-4\epsilon_{s_3s_2}^{P_2},$$
$$e_1(y_3^{L_2},y_4^{L_2},y_5^{L_2})\mapsto \epsilon_{s_2}^{P_2},$$
$$e_2(y_3^{L_2},y_4^{L_2},y_5^{L_2})\mapsto -5\epsilon_{s_1s_2}^{P_2}+6\epsilon_{s_3s_2}^{P_2},$$
$$e_3(y_3^{L_2},y_4^{L_2},y_5^{L_2})\mapsto -4\epsilon_{s_3s_1s_2}^{P_2}-10\epsilon_{s_2s_3s_2}^{P_2}-36\epsilon_{s_4s_3s_2}^{P_2}.$$
(c) ($L_3$):  Let 
$$y_1^{L_3}=2\Theta_1-\Theta_3,\ y_2^{L_3}=-2\Theta_1+2\Theta_2-\Theta_3,\ y_3^{L_3}=-2\Theta_2+3\Theta_3,\ y_4^{L_3}=2\Theta_3-3\Theta_4,\ y_5^{L_3}=-\Theta_3+3\Theta_4.$$
Then,
$$\Rep^\C_{\poly}(L_3)
=\C_{\sym}[y_1^{L_3},y_2^{L_3}, y_3^{L_3}] \otimes_\C  \C_{\sym}\left[y_4^{L_3}, y_5^{L_3} \right]\Big/\left(y_1^{L_3}+y_2^{L_3}+y_3^{L_3}-y_4^{L_3}-y_5^{L_3}\right).$$
Further, the  generators of $\Rep^\C_{\poly}(L_3)$ 
under $\xi^{P_3}$ go to:
$$e_1(y_1^{L_3},y_2^{L_3},y_3^{L_3})\mapsto \epsilon_{s_3}^{P_3},$$
$$e_2(y_1^{L_3},y_2^{L_3},y_3^{L_3})\mapsto 3\epsilon_{s_2s_3}^{P_3}-5\epsilon_{s_4s_3}^{P_3},$$
$$e_3(y_1^{L_3},y_2^{L_3},y_3^{L_3})\mapsto 11\epsilon_{s_1s_2s_3}^{P_3}-5\epsilon_{s_3s_2s_3}^{P_3}-2\epsilon_{s_4s_2s_3}^{P_3},$$
$$e_1(y_4^{L_3},y_5^{L_3})\mapsto \epsilon_{s_3}^{P_3},$$
$$e_2(y_4^{L_3},y_5^{L_3})\mapsto -2\epsilon_{s_2s_3}^{P_3}+7\epsilon_{s_4s_3}^{P_3}.$$
(d) ($L_4$):  Let 
$$y_1^{L_4}=\Theta_1-\Theta_4,\ y_2^{L_4}=-\Theta_1+\Theta_2-\Theta_4,\ y_3^{L_4}=-\Theta_2+2\Theta_3-\Theta_4,\ y_4^{L_4}=\Theta_4.$$
Then,
$$\Rep^\C_{\poly}(L_4)
=\C_{\sym}\left[(y_1^{L_4})^2, (y_2^{L_4})^2, (y_3^{L_4})^2 \right] \otimes_\C \C[y_4^{L_4}].$$
Further, the  generators of $\Rep^\C_{\poly}(L_4)$ 
under $\xi^{P_4}$ go to:
$$e_1\left((y_1^{L_4})^2,(y_2^{L_4})^2,(y_3^{L_4})^2\right)\mapsto -\epsilon_{s_3s_4}^{P_4},$$
$$e_2\left((y_1^{L_4})^2,(y_2^{L_4})^2,(y_3^{L_4})^2\right)\mapsto 7\epsilon_{s_1s_2s_3s_4}^{P_4}-5\epsilon_{s_3s_2s_3s_4}^{P_4},$$
$$e_3\left((y_1^{L_4})^2,(y_2^{L_4})^2,(y_3^{L_4})^2\right)\mapsto -2\epsilon_{s_2s_3s_1s_2s_3s_4}^{P_4}+3\epsilon_{s_4s_3s_1s_2s_3s_4}^{P_4},$$
$$y_4^{L_4}\mapsto \epsilon_{s_4}^{P_4}.$$
\end{theorem}       

\section{$E_6$}
We use $\{\omega_1, \dots,\omega_6\}$ as a basis for $\mf[t]^*$. This gives rise to a coordinate system on 
$T$ as earlier using the identity \eqref{neweqn6}.

\begin{lemma} \label{lem5.1} \begin{equation}
    \theta_{\omega_1}(t_1, \dots , t_6)= \left(\Theta_1,\dots, \Theta_6\right),
   \end{equation} 
   where 
   \begin{align*}\Theta_1(t)= & \frac{1}{6}\big( 4t_1^{1}+t_1^{-1}t_3^{1}+t_3^{-1}t_4^{1}+t_2^{1}t_4^{-1}t_5^{1}+t_2^{-1}t_5^{1}+t_2^{1}t_5^{-1}t_6^{1}
+t_2^{-1}t_4^{1}t_5^{-1}t_6^{1}\\
&+t_2^{1}t_6^{-1}+t_3^{1}t_4^{-1}t_6^{1}+t_2^{-1}t_4^{1}t_6^{-1}+t_1^{1}t_3^{-1}t_6^{1}+t_3^{1}t_4^{-1}t_5^{1}t_6^{-1}
-2t_1^{-1}t_6^{1}
+t_1^{1}t_3^{-1}t_5^{1}t_6^{-1}\\
&+t_3^{1}t_5^{-1}-2t_1^{-1}t_5^{1}t_6^{-1}+t_1^{1}t_3^{-1}t_4^{1}t_5^{-1}-2t_1^{-1}t_4^{1}t_5^{-1}
+t_1^{1}t_2^{1}t_4^{-1}
-2t_1^{-1}t_2^{1}t_3^{1}t_4^{-1}\\
&+t_1^{1}t_2^{-1}-2t_1^{-1}t_2^{-1}t_3^{1}-2t_2^{1}t_3^{-1}-2t_2^{-1}t_3^{-1}t_4^{1}
-2t_4^{-1}t_5^{1}-2t_5^{-1}t_6^{1}-2t_6^{-1} \big).
\end{align*}
\begin{align*}
    \Theta_2(t)= & \frac{1}{6}\big( 3t_1^{1}+3t_1^{-1}t_3^{1}+3t_3^{-1}t_4^{1}+3t_2^{1}t_4^{-1}t_5^{1}+3t_2^{1}t_5^{-1}t_6^{1}+3t_2^{1}t_6^{-1}
-3t_1^{1}t_2^{-1}
-3t_1^{-1}t_2^{-1}t_3^{1}\\
&-3t_2^{-1}t_3^{-1}t_4^{1}-3t_4^{-1}t_5^{1}-3t_5^{-1}t_6^{1}-3t_6^{-1} \big).
\end{align*}
\begin{align*}
    \Theta_3(t)= & \frac{1}{6}\big( 5t_1^{1}+5t_1^{-1}t_3^{1}+2t_3^{-1}t_4^{1}+2t_2^{1}t_4^{-1}t_5^{1}+2t_2^{-1}t_5^{1}+2t_2^{1}t_5^{-1}t_6^{1}\\
&+2t_2^{-1}t_4^{1}t_5^{-1}t_6^{1}+2t_2^{1}t_6^{-1}+2t_3^{1}t_4^{-1}t_6^{1}+2t_2^{-1}t_4^{1}t_6^{-1}-t_1^{1}t_3^{-1}t_6^{1}+2t_3^{1}t_4^{-1}t_5^{1}t_6^{-1}\\
&-t_1^{-1}t_6^{1}-t_1^{1}t_3^{-1}t_5^{1}t_6^{-1}+2t_3^{1}t_5^{-1}-t_1^{-1}t_5^{1}t_6^{-1}-t_1^{1}t_3^{-1}t_4^{1}t_5^{-1}-t_1^{-1}t_4^{1}t_5^{-1}\\
&-t_1^{1}t_2^{1}t_4^{-1}-t_1^{-1}t_2^{1}t_3^{1}t_4^{-1}-t_1^{1}t_2^{-1}-t_1^{-1}t_2^{-1}t_3^{1}-4t_2^{1}t_3^{-1}-4t_2^{-1}t_3^{-1}t_4^{1}\\
&-4t_4^{-1}t_5^{1}-4t_5^{-1}t_6^{1}-4t_6^{-1} \big).
\end{align*}
\begin{align*}
    \Theta_4(t)= & \frac{1}{6}\big( 6t_1^{1}+6t_1^{-1}t_3^{1}+6t_3^{-1}t_4^{1}+3t_2^{1}t_4^{-1}t_5^{1}+3t_2^{-1}t_5^{1}+3t_2^{1}t_5^{-1}t_6^{1}\\
&+3t_2^{-1}t_4^{1}t_5^{-1}t_6^{1}+3t_2^{1}t_6^{-1}+3t_2^{-1}t_4^{1}t_6^{-1}-3t_1^{1}t_2^{1}t_4^{-1}-3t_1^{-1}t_2^{1}t_3^{1}t_4^{-1}-3t_1^{1}t_2^{-1}\\
&-3t_1^{-1}t_2^{-1}t_3^{1}-3t_2^{1}t_3^{-1}-3t_2^{-1}t_3^{-1}t_4^{1}-6t_4^{-1}t_5^{1}-6t_5^{-1}t_6^{1}-6t_6^{-1} \big).
\end{align*}
\begin{align*}
    \Theta_5(t)= & \frac{1}{6}\big( 4t_1^{1}+4t_1^{-1}t_3^{1}+4t_3^{-1}t_4^{1}+4t_2^{1}t_4^{-1}t_5^{1}+4t_2^{-1}t_5^{1}+t_2^{1}t_5^{-1}t_6^{1}\\
&+t_2^{-1}t_4^{1}t_5^{-1}t_6^{1}+t_2^{1}t_6^{-1}+t_3^{1}t_4^{-1}t_6^{1}+t_2^{-1}t_4^{1}t_6^{-1}+t_1^{1}t_3^{-1}t_6^{1}+t_3^{1}t_4^{-1}t_5^{1}t_6^{-1}\\
&+t_1^{-1}t_6^{1}+t_1^{1}t_3^{-1}t_5^{1}t_6^{-1}-2t_3^{1}t_5^{-1}+t_1^{-1}t_5^{1}t_6^{-1}-2t_1^{1}t_3^{-1}t_4^{1}t_5^{-1}-2t_1^{-1}t_4^{1}t_5^{-1}\\
&-2t_1^{1}t_2^{1}t_4^{-1}-2t_1^{-1}t_2^{1}t_3^{1}t_4^{-1}-2t_1^{1}t_2^{-1}-2t_1^{-1}t_2^{-1}t_3^{1}-2t_2^{1}t_3^{-1}-2t_2^{-1}t_3^{-1}t_4^{1}\\
&-2t_4^{-1}t_5^{1}-5t_5^{-1}t_6^{1}-5t_6^{-1} \big).
\end{align*}
\begin{align*}
    \Theta_6(t)= & \frac{1}{6}\big( 2t_1^{1}+2t_1^{-1}t_3^{1}+2t_3^{-1}t_4^{1}+2t_2^{1}t_4^{-1}t_5^{1}+2t_2^{-1}t_5^{1}+2t_2^{1}t_5^{-1}t_6^{1}\\
&+2t_2^{-1}t_4^{1}t_5^{-1}t_6^{1}-t_2^{1}t_6^{-1}+2t_3^{1}t_4^{-1}t_6^{1}-t_2^{-1}t_4^{1}t_6^{-1}+2t_1^{1}t_3^{-1}t_6^{1}-t_3^{1}t_4^{-1}t_5^{1}t_6^{-1}\\
&+2t_1^{-1}t_6^{1}-t_1^{1}t_3^{-1}t_5^{1}t_6^{-1}-t_3^{1}t_5^{-1}-t_1^{-1}t_5^{1}t_6^{-1}-t_1^{1}t_3^{-1}t_4^{1}t_5^{-1}-t_1^{-1}t_4^{1}t_5^{-1}\\
&-t_1^{1}t_2^{1}t_4^{-1}-t_1^{-1}t_2^{1}t_3^{1}t_4^{-1}-t_1^{1}t_2^{-1}-t_1^{-1}t_2^{-1}t_3^{1}-t_2^{1}t_3^{-1}-t_2^{-1}t_3^{-1}t_4^{1}\\
&-t_4^{-1}t_5^{1}-t_5^{-1}t_6^{1}-4t_6^{-1} \big).
\end{align*}
\end{lemma}
Combining the equation \eqref{eqnborel} with Lemma \ref{lem5.1} and using Theorem \ref{duan2005multiplicative}, we get the following:
\begin{theorem}  The generators of $\Rep^\C_{\poly}(L_r)$ under $\xi^{P_r}$ are mapped as follows:

(a) ($L_1$): For $L_1$, we let 
$$y_1^{L_1}=3\Theta_{1},\ y_2^{L_1}=-\Theta_{1}-2\Theta_{2}+2\Theta_{3},\ y_3^{L_1}=-\Theta_{1}+2\Theta_{2}+2\Theta_{3}-2\Theta_{4},$$
$$y_4^{L_1}=-\Theta_{1}+2\Theta_{4}-2\Theta_{5},\ y_5^{L_1}=-\Theta_{1}+2\Theta_{5}-2\Theta_{6},\ y_6^{L_1}=-\Theta_{1}+2\Theta_{6}.$$
Then, 
$$\Rep^\C_{\poly}(L_1)
=\C[y_1^{L_1}]\otimes_\C\left[R_{6,1}\oplus (y_2^{L_1}y_3^{L_1}y_4^{L_1}y_5^{L_1}y_6^{L_1})R_{6,1}\right],$$
where $R_{6,1}=\C_{\sym}\left[(y_2^{L_1})^2, (y_3^{L_1})^2, (y_4^{L_1})^2 , (y_5^{L_1})^2, (y_6^{L_1})^2 \right].$\\
Further, the generators of $\Rep^\C_{\poly}(L_1)$ under $\xi^{P_1}$ go to 
$$y_1^{L_1}\to 3\epsilon_{s_1}^{P_1},$$
$$e_1\left((y_2^{L_1})^2, (y_3^{L_1})^2, (y_4^{L_1})^2, (y_5^{L_1})^2, (y_6^{L_1})^2\right)\mapsto -3\epsilon_{s_3s_1}^{P_1},$$
$$e_2\left((y_2^{L_1})^2, (y_3^{L_1})^2, (y_4^{L_1})^2, (y_5^{L_1})^2, (y_6^{L_1})^2\right)\mapsto -54\epsilon_{s_2s_4s_3s_1}^{P_1}+42\epsilon_{s_5s_4s_3s_1}^{P_1},$$
$$e_3\left((y_2^{L_1})^2, (y_3^{L_1})^2, (y_4^{L_1})^2, (y_5^{L_1})^2, (y_6^{L_1})^2\right)\mapsto -108\epsilon_{s_4s_5s_2s_4s_3s_1}^{P_1}+174\epsilon_{s_6s_5s_2s_4s_3s_1}^{P_1},$$
$$e_4\left((y_2^{L_1})^2, (y_3^{L_1})^2, (y_4^{L_1})^2, (y_5^{L_1})^2, (y_6^{L_1})^2\right)\mapsto 
42\epsilon_{s_1s_3s_4s_5s_2s_4s_3s_1}^{P_1}+291\epsilon_{s_6s_3s_4s_5s_2s_4s_3s_1}^{P_1}-519\epsilon_{s_5s_6s_4s_5s_2s_4s_3s_1}^{P_1},$$
$$e_5\left((y_2^{L_1})^2, (y_3^{L_1})^2, (y_4^{L_1})^2, (y_5^{L_1})^2, (y_6^{L_1})^2\right)\mapsto 333\epsilon_{s_5s_6s_1s_3s_4s_5s_2s_4s_3s_1}^{P_1}-180\epsilon_{s_4s_5s_6s_3s_4s_5s_2s_4s_3s_1}^{P_1}.$$
$$e_5\left(y_2^{L_1}, y_3^{L_1}, y_4^{L_1}, y_5^{L_1}, y_6^{L_1}\right)\mapsto 6\epsilon_{s_5s_2s_4s_3s_1}^{P_1} -21\epsilon_{s_6s_5s_4s_3s_1}^{P_1}$$

(b)  ($L_2$): For $L_2$, we let 
$$y_1^{L_2}=3\Theta_{1}-\Theta_{2},\ y_2^{L_2}=-3\Theta_{1}-\Theta_{2}+3\Theta_{3},\ y_3^{L_2}=-\Theta_{2}-3\Theta_{3}+3\Theta_{4},$$
$$y_4^{L_2}=2\Theta_{2}-3\Theta_{4}+3\Theta_{5},\ y_5^{L_2}=2\Theta_{2}-3\Theta_{5}+3\Theta_{6},\ y_6^{L_2}=2\Theta_{2}-3\Theta_{6}.$$
Then,
$$ \Rep^\C_{\poly}(L_2)
=\C_{\sym}[y_1^{L_2},y_2^{L_2}, y_3^{L_2}, y_4^{L_2}, y_5^{L_2}, y_6^{L_2}].$$
Further, the generators of $\Rep^\C_{\poly}(L_2)$ under $\xi^{P_2}$ go to 
$$e_1\left( y_1^{L_2}, y_2^{L_2}, y_3^{L_2}, y_4^{L_2}, y_5^{L_2}, y_6^{L_2}\right)\mapsto 3\epsilon_{s_2}^{P_2},$$
$$e_2\left( y_1^{L_2}, y_2^{L_2}, y_3^{L_2}, y_4^{L_2}, y_5^{L_2}, y_6^{L_2}\right)\mapsto 6\epsilon_{s_4s_2}^{P_2},$$
$$e_3\left( y_1^{L_2}, y_2^{L_2}, y_3^{L_2}, y_4^{L_2}, y_5^{L_2}, y_6^{L_2}\right)\mapsto 34\epsilon_{s_3s_4s_2}^{P_2} - 20\epsilon_{s_5s_4s_2}^{P_2},$$
$$e_4\left( y_1^{L_2}, y_2^{L_2}, y_3^{L_2}, y_4^{L_2}, y_5^{L_2}, y_6^{L_2}\right)\mapsto 141\epsilon_{s_1s_3s_4s_2}^{P_2} - 42\epsilon_{s_5s_3s_4s_2}^{P_2} + 60\epsilon_{s_6s_5s_4s_2}^{P_2},$$
$$e_5\left( y_1^{L_2}, y_2^{L_2}, y_3^{L_2}, y_4^{L_2}, y_5^{L_2}, y_6^{L_2}\right)\mapsto 
9\epsilon_{s_5s_1s_3s_4s_2}^{P_2} - 48\epsilon_{s_4s_5s_3s_4s_2}^{P_2} + 90\epsilon_{s_6s_5s_3s_4s_2}^{P_2},$$
$$e_6\left( y_1^{L_2}, y_2^{L_2}, y_3^{L_2}, y_4^{L_2}, y_5^{L_2}, y_6^{L_2}\right)\mapsto $$
$$-130\epsilon_{s_4s_5s_1s_3s_4s_2}^{P_2} + 411\epsilon_{s_6s_5s_1s_3s_4s_2}^{P_2} + 56\epsilon_{s_2s_4s_5s_3s_4s_2}^{P_2} -76\epsilon_{s_6s_4s_5s_3s_4s_2}^{P_2}.$$

(c)  ($L_3$): For $L_3$, we let 
$$y_1^{L_3}=5\Theta_{1}-\Theta_{3},\ y_2^{L_3}=-5\Theta_{1}+4\Theta_{3},\ y_3^{L_3}=-4\Theta_{2}+3\Theta_{3},\ y_4^{L_3}=4\Theta_{2}+3\Theta_{3}-4\Theta_{4},$$
$$y_5^{L_3}=-\Theta_{3}+4\Theta_{4}-4\Theta_{5},\ y_6^{L_3}=-\Theta_{3}+4\Theta_{5}-4\Theta_{6},\ y_7^{L_3}=-\Theta_{3}+4\Theta_{6}.$$
Then,
$$ \Rep^\C_{\poly}(L_3)
=\C_{\sym}[y_1^{L_3},y_2^{L_3}] \otimes_\C \C_{\sym}[y_3^{L_3}, y_4^{L_3}, y_5^{L_3}, y_6^{L_3}, y_7^{L_3}]/R_{6,3},$$
where $R_{6,3}=y_1^{L_3}+y_2^{L_3}-y_3^{L_3}-y_4^{L_3}-y_5^{L_3}-y_6^{L_3}-y_7^{L_3}$.\\
Further, the generators of $\Rep^\C_{\poly}(L_3)$ under $\xi^{P_3}$ go to 
$$e_1\left( y_1^{L_3}, y_2^{L_3}\right)\mapsto 3\epsilon_{s_3}^{P_3},$$
$$e_2\left( y_1^{L_3}, y_2^{L_3}\right)\mapsto 21\epsilon_{s_1s_3}^{P_3} - 4\epsilon_{s_4s_3}^{P_3},$$
$$e_1\left( y_3^{L_3}, y_4^{L_3}, y_5^{L_3}, y_6^{L_3}, y_7^{L_3}\right)\mapsto 4\epsilon_{s_3}^{P_3},$$
$$e_2\left( y_3^{L_3}, y_4^{L_3}, y_5^{L_3}, y_6^{L_3}, y_7^{L_3}\right)\mapsto -6\epsilon_{s_1s_3}^{P_3} +10\epsilon_{s_4s_3}^{P_3},$$
$$e_3\left( y_3^{L_3}, y_4^{L_3}, y_5^{L_3}, y_6^{L_3}, y_7^{L_3}\right)\mapsto -4\epsilon_{s_4s_1s_3}^{P_3} -58\epsilon_{s_2s_4s_3}^{P_3} +70\epsilon_{s_5s_4s_3}^{P_3},$$
$$e_4\left( y_3^{L_3}, y_4^{L_3}, y_5^{L_3}, y_6^{L_3}, y_7^{L_3}\right)\mapsto$$ 
$$-31\epsilon_{s_2s_4s_1s_3}^{P_3} - 38\epsilon_{s_3s_4s_1s_3}^{P_3} + 31\epsilon_{s_5s_4s_1s_3}^{P_3} -118\epsilon_{s_5s_2s_4s_3}^{P_3} +325\epsilon_{s_6s_5s_4s_3}^{P_3}.$$
$$e_5\left( y_3^{L_3}, y_4^{L_3}, y_5^{L_3}, y_6^{L_3}, y_7^{L_3}\right)\mapsto$$
$$35\epsilon_{s_3s_2s_4s_1s_3}^{P_3} - 40\epsilon_{s_5s_2s_4s_1s_3}^{P_3} - 93\epsilon_{s_5s_3s_4s_1s_3}^{P_3}492\epsilon_{s_6s_5s_4s_1s_3}^{P_3} +78\epsilon_{s_4s_5s_2s_4s_3}^{P_3} -331\epsilon_{s_6s_5s_2s_4s_3}^{P_3}.$$

(d)  ($L_4$): For $L_4$, we let 
$$y_1^{L_4}=6\Theta_{1}-\Theta_{4},\ y_2^{L_4}=-6\Theta_{1}+6\Theta_{3}-\Theta_{4},\ y_3^{L_4}=-6\Theta_{3}+5\Theta_{4},\ y_4^{L_4}=-\Theta_{4}+6\Theta_{6},$$
$$y_5^{L_4}=-\Theta_{4}+6\Theta_{5}-6\Theta_{6},\ y_6^{L_4}=5\Theta_{4}-6\Theta_{5},\ y_7^{L_4}=3\Theta_{2}-\Theta_{4}, y_8^{L_4}=-3\Theta_{2}+2\Theta_{4}.$$
Then,
$$ \Rep^\C_{\poly}(L_4)
=\C_{\sym}[y_1^{L_4},y_2^{L_4},y_3^{L_4}] \otimes_\C \C_{\sym}[y_4^{L_4}, y_5^{L_4}, y_6^{L_4}] \otimes_\C \C_{\sym}[y_7^{L_4}, y_8^{L_4}]/(R_{6,4},R_{6,4}'),$$
where $R_{6,4}=y_1^{L_4}+y_2^{L_4}+y_3^{L_4}-3y_7^{L_4}-3y_8^{L_4}$, and $R_{6,4}'=y_4^{L_4}+y_5^{L_4}+y_6^{L_4}-3y_7^{L_4}-3y_8^{L_4}$.\\
Further, the generators of $\Rep^\C_{\poly}(L_4)$ under $\xi^{P_4}$ go to 
$$e_1\left( y_1^{L_4}, y_2^{L_4}, y_3^{L_4}\right)\mapsto 3\epsilon_{s_4}^{P_4},$$
$$e_2\left( y_1^{L_4}, y_2^{L_4}, y_3^{L_4}\right)\to -9\epsilon_{s_2s_4}^{P_4} + 27\epsilon_{s_3s_4}^{P_4} -9\epsilon_{s_5s_4}^{P_4},$$
$$e_3\left( y_1^{L_4}, y_2^{L_4}, y_3^{L_4}\right)\mapsto
-26\epsilon_{s_3s_2s_4}^{P_4} + 10\epsilon_{s_5s_2s_4}^{P_4} + 185\epsilon_{s_1s_3s_4}^{P_4} -26\epsilon_{s_5s_3s_4}^{P_4} + 5\epsilon_{s_6s_5s_4}^{P_4},$$
$$e_1\left( y_4^{L_4}, y_5^{L_4}, y_6^{L_4}\right)\mapsto 3\epsilon_{s_4}^{P_4},$$
$$e_2\left( y_4^{L_4}, y_5^{L_4}, y_6^{L_4}\right)\mapsto -9\epsilon_{s_2s_4}^{P_4} - 9\epsilon_{s_3s_4}^{P_4} +27\epsilon_{s_5s_4}^{P_4},$$
$$e_3\left( y_4^{L_4}, y_5^{L_4}, y_6^{L_4}\right)\mapsto
10\epsilon_{s_3s_2s_4}^{P_4} - 26\epsilon_{s_5s_2s_4}^{P_4} + 5\epsilon_{s_1s_3s_4}^{P_4} -26\epsilon_{s_5s_3s_4}^{P_4} + 185\epsilon_{s_6s_5s_4}^{P_4},$$
$$e_1\left( y_7^{L_4}, y_8^{L_4}\right)\mapsto \epsilon_{s_4}^{P_4},$$
$$e_2\left( y_7^{L_4}, y_8^{L_4}\right)\mapsto 7\epsilon_{s_2s_4}^{P_4} -2\epsilon_{s_3s_4}^{P_4} -2\epsilon_{s_5s_4}^{P_4}.$$

(e)  ($L_5$): For $L_5$, we let 
$$y_1^{L_5}=-\Theta_{5}+5\Theta_{6},\ y_2^{L_5}=4\Theta_{5}-5\Theta_{6},\ y_3^{L_5}=-4\Theta_{2}+3\Theta_{5},\ y_4^{L_5}=4\Theta_{2}-4\Theta_{4}+3\Theta_{5},$$
$$y_5^{L_5}=-4\Theta_{3}+4\Theta_{4}-\Theta_{5},\ y_6^{L_5}=-4\Theta_{1}+4\Theta_{3}-\Theta_{5},\ y_7^{L_5}=4\Theta_{1}-\Theta_{5}.$$
Then,
$$  \Rep^\C_{\poly}(L_5)
=\C_{\sym}[y_1^{L_5},y_2^{L_5}] \otimes_\C \C_{\sym}[y_3^{L_5}, y_4^{L_5}, y_5^{L_5}, y_6^{L_5}, y_7^{L_5}]/R_{6,5},$$
where $R_{6,5}=y_1^{L_5}+y_2^{L_5}-y_3^{L_5}-y_4^{L_5}- y_5^{L_5}- y_6^{L_5}- y_7^{L_5} $.\\
Further, the generators of $\Rep^\C_{\poly}(L_5)$ under $\xi^{P_5}$ go to 
$$e_1\left( y_1^{L_5}, y_2^{L_5}\right)\mapsto 3\epsilon_{s_5}^{P_5},$$
$$e_2\left( y_1^{L_5}, y_2^{L_5}\right)\mapsto -4\epsilon_{s_4s_5}^{P_5} + 21\epsilon_{s_6s_5}^{P_5},$$
$$e_1\left( y_3^{L_5}, y_4^{L_5}, y_5^{L_5}, y_6^{L_5}, y_7^{L_5}\right)\mapsto 3\epsilon_{s_5}^{P_5},$$
$$e_2\left( y_3^{L_5}, y_4^{L_5}, y_5^{L_5}, y_6^{L_5}, y_7^{L_5}\right)\mapsto 10\epsilon_{s_4s_5}^{P_5} -6 \epsilon_{s_6s_5}^{P_5},$$
$$e_3\left( y_3^{L_5}, y_4^{L_5}, y_5^{L_5}, y_6^{L_5}, y_7^{L_5}\right)\mapsto -58\epsilon_{s_2s_4s_5}^{P_5} +70 \epsilon_{s_3s_4s_5}^{P_5} -4\epsilon_{s_6s_4s_5}^{P_5},$$
$$e_4\left( y_3^{L_5}, y_4^{L_5}, y_5^{L_5}, y_6^{L_5}, y_7^{L_5}\right)\mapsto$$
$$-118\epsilon_{s_3s_2s_4s_5}^{P_5} +31 \epsilon_{s_6s_2s_4s_5}^{P_5} +325\epsilon_{s_1s_3s_4s_5}^{P_5} +31\epsilon_{s_6s_3s_4s_5}^{P_5} -38\epsilon_{s_5s_6s_4s_5}^{P_5},$$
$$e_5\left( y_3^{L_5}, y_4^{L_5}, y_5^{L_5}, y_6^{L_5}, y_7^{L_5}\right)\mapsto$$
$$-331\epsilon_{s_1s_3s_2s_4s_5}^{P_5} +78 \epsilon_{s_4s_3s_2s_4s_5}^{P_5} -40\epsilon_{s_6s_3s_2s_4s_5}^{P_5} +35\epsilon_{s_5s_6s_2s_4s_5}^{P_5} +492\epsilon_{s_6s_1s_3s_4s_5}^{P_5} -93\epsilon_{s_5s_6s_3s_4s_5}^P.$$
(f)  ($L_6$): For $L_6$, we let
$$y_1^{L_6}=3\Theta_{6},\ y_2^{L_6}=2\Theta_{1}-\Theta_{6},\ y_3^{L_6}=-2\Theta_{1}+2\Theta_{3}-\Theta_{6},$$
$$y_4^{L_6}=-2\Theta_{3}+2\Theta_{4}-\Theta_{6},\ y_5^{L_6}=2\Theta_{2}-2\Theta_{4}+2\Theta_{5}-\Theta_{6},\ y_6^{L_6}=-2\Theta_{2}+2\Theta_{5}-\Theta_{6}.$$
Then,
$$  \Rep^\C_{\poly}(L_6)
=\C[y_1^{L_6}]\otimes_\C\left[R\oplus (y_2^{L_6}y_3^{L_6}y_4^{L_6}y_5^{L_6}y_6^{L_6})R\right],$$
where $R=\C_{\sym}\left[(y_2^{L_6})^2, (y_3^{L_6})^2, (y_4^{L_6})^2 , (y_5^{L_6})^2, (y_6^{L_6})^2 \right].$\\
Further, the generators of $\Rep^\C_{\poly}(L_6)$ under $\xi^{P_6}$ go to 
$$y_1^{L_6}\mapsto 3\epsilon_{s_6}^{P_6},$$
$$e_1\left((y_2^{L_6})^2, (y_3^{L_6})^2, (y_4^{L_6})^2 , (y_5^{L_6})^2, (y_6^{L_6})^2\right)\mapsto -5\epsilon_{s_5s_6}^{P_6},$$
$$e_2\left((y_2^{L_6})^2, (y_3^{L_6})^2, (y_4^{L_6})^2 , (y_5^{L_6})^2, (y_6^{L_6})^2\right)\mapsto -54\epsilon_{s_2s_4s_5s_6}^{P_6}+4254\epsilon_{s_3s_4s_5s_6}^{P_6},$$
$$e_3\left((y_2^{L_6})^2, (y_3^{L_6})^2, (y_4^{L_6})^2 , (y_5^{L_6})^2, (y_6^{L_6})^2\right)\mapsto
174\epsilon_{s_1s_3s_2s_4s_5s_6}^{P_6}-108\epsilon_{s_4s_3s_2s_4s_5s_6}^{P_6},$$
$$e_4\left((y_2^{L_6})^2, (y_3^{L_6})^2, (y_4^{L_6})^2 , (y_5^{L_6})^2, (y_6^{L_6})^2\right)\mapsto $$
$$-519\epsilon_{s_3s_4s_1s_3s_2s_4s_5s_6}^{P_6}+291\epsilon_{s_5s_4s_1s_3s_2s_4s_5s_6}^{P_6}+42\epsilon_{s_6s_5s_4s_3s_2s_4s_5s_6}^{P_6},$$
$$e_5\left((y_2^{L_6})^2, (y_3^{L_6})^2, (y_4^{L_6})^2 , (y_5^{L_6})^2, (y_6^{L_6})^2\right)\mapsto
-180\epsilon_{s_4s_5s_3s_4s_1s_3s_2s_4s_5s_6}^{P_6}+333\epsilon_{s_6s_5s_3s_4s_1s_3s_2s_4s_5s_6}^{P_6},$$
\end{theorem}

\section{$E_7$}
We use $\{\omega_1, \dots,\omega_7\}$ as a basis for $\mf[t]^*$. This gives rise to a coordinate system on 
$T$ as earlier using the identity \eqref{neweqn6}.

\begin{lemma} \label{lem6.1} \begin{equation}
    \theta_{\omega_7}(t_1, \dots , t_7)= \left(\Theta_1,\dots, \Theta_7\right),
   \end{equation} 
   where 
\begin{align*}
    \Theta_1(t)=&\frac{1}{24}\big( 2t_7^{1}+2t_6^{1}t_7^{-1}+2t_5^{1}t_6^{-1}+2t_4^{1}t_5^{-1}+2t_2^{1}t_3^{1}t_4^{-1}
+2t_2^{-1}t_3^{1}+2t_1^{1}t_2^{1}t_3^{-1}
\\
&+2t_1^{1}t_2^{-1}t_3^{-1}t_4^{1}+2t_1^{1}t_4^{-1}t_5^{1}+2t_1^{1}t_5^{-1}t_6^{1}
+2t_1^{1}t_6^{-1}t_7^{1}+2t_1^{1}t_7^{-1}-2t_1^{-1}t_7^{1}-2t_1^{-1}t_6^{1}t_7^{-1}\\
&-2t_1^{-1}t_5^{1}t_6^{-1}
-2t_1^{-1}t_4^{1}t_5^{-1}-2t_1^{-1}t_2^{1}t_3^{1}t_4^{-1}-2t_1^{-1}t_2^{-1}t_3^{1}-2t_2^{1}t_3^{-1}-2t_2^{-1}t_3^{-1}t_4^{1}\\
&-2t_4^{-1}t_5^{1}-2t_5^{-1}t_6^{1}-2t_6^{-1}t_7^{1}-2t_7^{-1}\big)
\end{align*}
\begin{align*}
    \Theta_2(t)=&\frac{1}{24}\big( 3t_7^{1}+3t_6^{1}t_7^{-1}+3t_5^{1}t_6^{-1}+3t_4^{1}t_5^{-1}+3t_2^{1}t_3^{1}t_4^{-1}
+t_2^{-1}t_3^{1}+3t_1^{1}t_2^{1}t_3^{-1}\\
&+t_1^{1}t_2^{-1}t_3^{-1}t_4^{1}+3t_1^{-1}t_2^{1}+t_1^{-1}t_2^{-1}t_4^{1}
+t_1^{1}t_4^{-1}t_5^{1}+t_1^{-1}t_3^{1}t_4^{-1}t_5^{1}+t_1^{1}t_5^{-1}t_6^{1}+t_3^{-1}t_5^{1}\\
&+t_1^{-1}t_3^{1}t_5^{-1}t_6^{1}
+t_1^{1}t_6^{-1}t_7^{1}+t_3^{-1}t_4^{1}t_5^{-1}t_6^{1}+t_1^{-1}t_3^{1}t_6^{-1}t_7^{1}+t_1^{1}t_7^{-1}+t_2^{1}t_4^{-1}t_6^{1}\\
&+t_3^{-1}t_4^{1}t_6^{-1}t_7^{1}+t_1^{-1}t_3^{1}t_7^{-1}-t_2^{-1}t_6^{1}+t_2^{1}t_4^{-1}t_5^{1}t_6^{-1}t_7^{1}+t_3^{-1}t_4^{1}t_7^{-1}
-t_2^{-1}t_5^{1}t_6^{-1}t_7^{1}\\
&+t_2^{1}t_5^{-1}t_7^{1}+t_2^{1}t_4^{-1}t_5^{1}t_7^{-1}-t_2^{-1}t_4^{1}t_5^{-1}t_7^{1}-t_2^{-1}t_5^{1}t_7^{-1}
+t_2^{1}t_5^{-1}t_6^{1}t_7^{-1}-t_3^{1}t_4^{-1}t_7^{1}\\
&-t_2^{-1}t_4^{1}t_5^{-1}t_6^{1}t_7^{-1}+t_2^{1}t_6^{-1}-t_1^{1}t_3^{-1}t_7^{1}
-t_3^{1}t_4^{-1}t_6^{1}t_7^{-1}-t_2^{-1}t_4^{1}t_6^{-1}-t_1^{-1}t_7^{1}-t_1^{1}t_3^{-1}t_6^{1}t_7^{-1}\\
&-t_3^{1}t_4^{-1}t_5^{1}t_6^{-1}
-t_1^{-1}t_6^{1}t_7^{-1}-t_1^{1}t_3^{-1}t_5^{1}t_6^{-1}-t_3^{1}t_5^{-1}-t_1^{-1}t_5^{1}t_6^{-1}-t_1^{1}t_3^{-1}t_4^{1}t_5^{-1}\\
&-t_1^{-1}t_4^{1}t_5^{-1}-t_1^{1}t_2^{1}t_4^{-1}-t_1^{-1}t_2^{1}t_3^{1}t_4^{-1}-3t_1^{1}t_2^{-1}-3t_1^{-1}t_2^{-1}t_3^{1}
-t_2^{1}t_3^{-1}\\
&-3t_2^{-1}t_3^{-1}t_4^{1}-3t_4^{-1}t_5^{1}-3t_5^{-1}t_6^{1}-3t_6^{-1}t_7^{1}
-3t_7^{-1} \big)
\end{align*}
\begin{align*}
    \Theta_3(t)=&\frac{1}{24}\big( 4t_7^{1}+4t_6^{1}t_7^{-1}+4t_5^{1}t_6^{-1}+4t_4^{1}t_5^{-1}+4t_2^{1}t_3^{1}t_4^{-1}
+4t_2^{-1}t_3^{1}+2t_1^{1}t_2^{1}t_3^{-1}\\
&+2t_1^{1}t_2^{-1}t_3^{-1}t_4^{1}+2t_1^{-1}t_2^{1}+2t_1^{-1}t_2^{-1}t_4^{1}
+2t_1^{1}t_4^{-1}t_5^{1}+2t_1^{-1}t_3^{1}t_4^{-1}t_5^{1}+2t_1^{1}t_5^{-1}t_6^{1}\\
&+2t_1^{-1}t_3^{1}t_5^{-1}t_6^{1}+2t_1^{1}t_6^{-1}t_7^{1}
+2t_1^{-1}t_3^{1}t_6^{-1}t_7^{1}+2t_1^{1}t_7^{-1}+2t_1^{-1}t_3^{1}t_7^{-1}-2t_1^{1}t_3^{-1}t_7^{1}-2t_1^{-1}t_7^{1}\\
&-2t_1^{1}t_3^{-1}t_6^{1}t_7^{-1}-2t_1^{-1}t_6^{1}t_7^{-1}-2t_1^{1}t_3^{-1}t_5^{1}t_6^{-1}-2t_1^{-1}t_5^{1}t_6^{-1}-2t_1^{1}t_3^{-1}t_4^{1}t_5^{-1}\\
&-2t_1^{-1}t_4^{1}t_5^{-1}-2t_1^{1}t_2^{1}t_4^{-1}-2t_1^{-1}t_2^{1}t_3^{1}t_4^{-1}-2t_1^{1}t_2^{-1}-2t_1^{-1}t_2^{-1}t_3^{1}\\
&-4t_2^{1}t_3^{-1}-4t_2^{-1}t_3^{-1}t_4^{1}-4t_4^{-1}t_5^{1}-4t_5^{-1}t_6^{1}-4t_6^{-1}t_7^{1}
-4t_7^{-1} \big)
\end{align*}
\begin{align*}
    \Theta_4(t)=&\frac{1}{24}\big( 6t_7^{1}+6t_6^{1}t_7^{-1}+6t_5^{1}t_6^{-1}+6t_4^{1}t_5^{-1}+4t_2^{1}t_3^{1}t_4^{-1}\\
&+4t_2^{-1}t_3^{1}+4t_1^{1}t_2^{1}t_3^{-1}+4t_1^{1}t_2^{-1}t_3^{-1}t_4^{1}+4t_1^{-1}t_2^{1}+4t_1^{-1}t_2^{-1}t_4^{1}\\
&+2t_1^{1}t_4^{-1}t_5^{1}+2t_1^{-1}t_3^{1}t_4^{-1}t_5^{1}+2t_1^{1}t_5^{-1}t_6^{1}+2t_3^{-1}t_5^{1}+2t_1^{-1}t_3^{1}t_5^{-1}t_6^{1}\\
&+2t_1^{1}t_6^{-1}t_7^{1}+2t_3^{-1}t_4^{1}t_5^{-1}t_6^{1}+2t_1^{-1}t_3^{1}t_6^{-1}t_7^{1}+2t_1^{1}t_7^{-1}+2t_3^{-1}t_4^{1}t_6^{-1}t_7^{1}\\
&+2t_1^{-1}t_3^{1}t_7^{-1}+2t_3^{-1}t_4^{1}t_7^{-1}-2t_3^{1}t_4^{-1}t_7^{1}-2t_1^{1}t_3^{-1}t_7^{1}-2t_3^{1}t_4^{-1}t_6^{1}t_7^{-1}\\
&-2t_1^{-1}t_7^{1}-2t_1^{1}t_3^{-1}t_6^{1}t_7^{-1}-2t_3^{1}t_4^{-1}t_5^{1}t_6^{-1}-2t_1^{-1}t_6^{1}t_7^{-1}-2t_1^{1}t_3^{-1}t_5^{1}t_6^{-1}\\
&-2t_3^{1}t_5^{-1}-2t_1^{-1}t_5^{1}t_6^{-1}-2t_1^{1}t_3^{-1}t_4^{1}t_5^{-1}-2t_1^{-1}t_4^{1}t_5^{-1}-4t_1^{1}t_2^{1}t_4^{-1}\\
&-4t_1^{-1}t_2^{1}t_3^{1}t_4^{-1}-4t_1^{1}t_2^{-1}-4t_1^{-1}t_2^{-1}t_3^{1}-4t_2^{1}t_3^{-1}-4t_2^{-1}t_3^{-1}t_4^{1}\\
&-6t_4^{-1}t_5^{1}-6t_5^{-1}t_6^{1}-6t_6^{-1}t_7^{1}-6t_7^{-1} \big)
\end{align*}
\begin{align*}
    \Theta_5(t)=&\frac{1}{24}\big( 5t_7^{1}+5t_6^{1}t_7^{-1}+5t_5^{1}t_6^{-1}+3t_4^{1}t_5^{-1}+3t_2^{1}t_3^{1}t_4^{-1}\\
&+3t_2^{-1}t_3^{1}+3t_1^{1}t_2^{1}t_3^{-1}+3t_1^{1}t_2^{-1}t_3^{-1}t_4^{1}+3t_1^{-1}t_2^{1}+3t_1^{-1}t_2^{-1}t_4^{1}\\
&+3t_1^{1}t_4^{-1}t_5^{1}+3t_1^{-1}t_3^{1}t_4^{-1}t_5^{1}+t_1^{1}t_5^{-1}t_6^{1}+3t_3^{-1}t_5^{1}+t_1^{-1}t_3^{1}t_5^{-1}t_6^{1}\\
&+t_1^{1}t_6^{-1}t_7^{1}+t_3^{-1}t_4^{1}t_5^{-1}t_6^{1}+t_1^{-1}t_3^{1}t_6^{-1}t_7^{1}+t_1^{1}t_7^{-1}+t_2^{1}t_4^{-1}t_6^{1}\\
&+t_3^{-1}t_4^{1}t_6^{-1}t_7^{1}+t_1^{-1}t_3^{1}t_7^{-1}+t_2^{-1}t_6^{1}+t_2^{1}t_4^{-1}t_5^{1}t_6^{-1}t_7^{1}+t_3^{-1}t_4^{1}t_7^{-1}\\
&+t_2^{-1}t_5^{1}t_6^{-1}t_7^{1}-t_2^{1}t_5^{-1}t_7^{1}+t_2^{1}t_4^{-1}t_5^{1}t_7^{-1}-t_2^{-1}t_4^{1}t_5^{-1}t_7^{1}+t_2^{-1}t_5^{1}t_7^{-1}\\
&-t_2^{1}t_5^{-1}t_6^{1}t_7^{-1}-t_3^{1}t_4^{-1}t_7^{1}-t_2^{-1}t_4^{1}t_5^{-1}t_6^{1}t_7^{-1}-t_2^{1}t_6^{-1}-t_1^{1}t_3^{-1}t_7^{1}\\
&-t_3^{1}t_4^{-1}t_6^{1}t_7^{-1}-t_2^{-1}t_4^{1}t_6^{-1}-t_1^{-1}t_7^{1}-t_1^{1}t_3^{-1}t_6^{1}t_7^{-1}-t_3^{1}t_4^{-1}t_5^{1}t_6^{-1}\\
&-t_1^{-1}t_6^{1}t_7^{-1}-t_1^{1}t_3^{-1}t_5^{1}t_6^{-1}-3t_3^{1}t_5^{-1}-t_1^{-1}t_5^{1}t_6^{-1}-3t_1^{1}t_3^{-1}t_4^{1}t_5^{-1}\\
&-3t_1^{-1}t_4^{1}t_5^{-1}-3t_1^{1}t_2^{1}t_4^{-1}-3t_1^{-1}t_2^{1}t_3^{1}t_4^{-1}-3t_1^{1}t_2^{-1}-3t_1^{-1}t_2^{-1}t_3^{1}\\
&-3t_2^{1}t_3^{-1}-3t_2^{-1}t_3^{-1}t_4^{1}-3t_4^{-1}t_5^{1}-5t_5^{-1}t_6^{1}-5t_6^{-1}t_7^{1}
-5t_7^{-1} \big)
\end{align*}
\begin{align*}
    \Theta_6(t)=&\frac{1}{24}\big( 4t_7^{1}+4t_6^{1}t_7^{-1}+2t_5^{1}t_6^{-1}+2t_4^{1}t_5^{-1}+2t_2^{1}t_3^{1}t_4^{-1}\\
&+2t_2^{-1}t_3^{1}+2t_1^{1}t_2^{1}t_3^{-1}+2t_1^{1}t_2^{-1}t_3^{-1}t_4^{1}+2t_1^{-1}t_2^{1}+2t_1^{-1}t_2^{-1}t_4^{1}\\
&+2t_1^{1}t_4^{-1}t_5^{1}+2t_1^{-1}t_3^{1}t_4^{-1}t_5^{1}+2t_1^{1}t_5^{-1}t_6^{1}+2t_3^{-1}t_5^{1}+2t_1^{-1}t_3^{1}t_5^{-1}t_6^{1}\\
&+2t_3^{-1}t_4^{1}t_5^{-1}t_6^{1}+2t_2^{1}t_4^{-1}t_6^{1}+2t_2^{-1}t_6^{1}-2t_2^{1}t_6^{-1}-2t_2^{-1}t_4^{1}t_6^{-1}\\
&-2t_3^{1}t_4^{-1}t_5^{1}t_6^{-1}-2t_1^{1}t_3^{-1}t_5^{1}t_6^{-1}-2t_3^{1}t_5^{-1}-2t_1^{-1}t_5^{1}t_6^{-1}-2t_1^{1}t_3^{-1}t_4^{1}t_5^{-1}\\
&-2t_1^{-1}t_4^{1}t_5^{-1}-2t_1^{1}t_2^{1}t_4^{-1}-2t_1^{-1}t_2^{1}t_3^{1}t_4^{-1}-2t_1^{1}t_2^{-1}-2t_1^{-1}t_2^{-1}t_3^{1}\\
&-2t_2^{1}t_3^{-1}-2t_2^{-1}t_3^{-1}t_4^{1}-2t_4^{-1}t_5^{1}-2t_5^{-1}t_6^{1}-4t_6^{-1}t_7^{1}
-4t_7^{-1} \big)
\end{align*}
\begin{align*}
    \Theta_7(t)=&\frac{1}{24}\big( 3t_7^{1}+t_6^{1}t_7^{-1}+t_5^{1}t_6^{-1}+t_4^{1}t_5^{-1}+t_2^{1}t_3^{1}t_4^{-1}
+t_2^{-1}t_3^{1}+t_1^{1}t_2^{1}t_3^{-1}\\
&+t_1^{1}t_2^{-1}t_3^{-1}t_4^{1}+t_1^{-1}t_2^{1}+t_1^{-1}t_2^{-1}t_4^{1}
+t_1^{1}t_4^{-1}t_5^{1}+t_1^{-1}t_3^{1}t_4^{-1}t_5^{1}+t_1^{1}t_5^{-1}t_6^{1}+t_3^{-1}t_5^{1}\\
&+t_1^{-1}t_3^{1}t_5^{-1}t_6^{1}
+t_1^{1}t_6^{-1}t_7^{1}+t_3^{-1}t_4^{1}t_5^{-1}t_6^{1}+t_1^{-1}t_3^{1}t_6^{-1}t_7^{1}-t_1^{1}t_7^{-1}+t_2^{1}t_4^{-1}t_6^{1}\\
&+t_3^{-1}t_4^{1}t_6^{-1}t_7^{1}-t_1^{-1}t_3^{1}t_7^{-1}+t_2^{-1}t_6^{1}+t_2^{1}t_4^{-1}t_5^{1}t_6^{-1}t_7^{1}-t_3^{-1}t_4^{1}t_7^{-1}
+t_2^{-1}t_5^{1}t_6^{-1}t_7^{1}\\
&+t_2^{1}t_5^{-1}t_7^{1}-t_2^{1}t_4^{-1}t_5^{1}t_7^{-1}+t_2^{-1}t_4^{1}t_5^{-1}t_7^{1}-t_2^{-1}t_5^{1}t_7^{-1}
-t_2^{1}t_5^{-1}t_6^{1}t_7^{-1}+t_3^{1}t_4^{-1}t_7^{1}\\
&-t_2^{-1}t_4^{1}t_5^{-1}t_6^{1}t_7^{-1}-t_2^{1}t_6^{-1}+t_1^{1}t_3^{-1}t_7^{1}
-t_3^{1}t_4^{-1}t_6^{1}t_7^{-1}-t_2^{-1}t_4^{1}t_6^{-1}+t_1^{-1}t_7^{1}\\
&-t_1^{1}t_3^{-1}t_6^{1}t_7^{-1}-t_3^{1}t_4^{-1}t_5^{1}t_6^{-1}
-t_1^{-1}t_6^{1}t_7^{-1}-t_1^{1}t_3^{-1}t_5^{1}t_6^{-1}-t_3^{1}t_5^{-1}-t_1^{-1}t_5^{1}t_6^{-1}\\
&-t_1^{1}t_3^{-1}t_4^{1}t_5^{-1}
-t_1^{-1}t_4^{1}t_5^{-1}-t_1^{1}t_2^{1}t_4^{-1}-t_1^{-1}t_2^{1}t_3^{1}t_4^{-1}-t_1^{1}t_2^{-1}-t_1^{-1}t_2^{-1}t_3^{1}\\
&-t_2^{1}t_3^{-1}-t_2^{-1}t_3^{-1}t_4^{1}-t_4^{-1}t_5^{1}-t_5^{-1}t_6^{1}-t_6^{-1}t_7^{1}
-3t_7^{-1} \big)
\end{align*}
\end{lemma}
Combining the equation \eqref{eqnborel} with Lemma \ref{lem6.1} and using Theorem \ref{duan2005multiplicative}, we get the following:
\begin{theorem}  The generators of $\Rep^\C_{\poly}(L_r)$ under $\xi^{P_r}$ are mapped as follows:

(a) ($L_1$): For $L_1$, we let 
$$y_1^{L_1}=\Theta_{1},\ y_2^{L_1}=-\Theta_{1}-2\Theta_{2}+2\Theta_{3},\ y_3^{L_1}=-\Theta_{1}+2\Theta_{2}+2\Theta_{3}-2\Theta_{4},$$
$$y_4^{L_1}=-\Theta_{1}+2\Theta_{4}-2\Theta_{5},\ y_5^{L_1}=-\Theta_{1}+2\Theta_{5}-2\Theta_{6},$$ $$y_6^{L_1}=-\Theta_{1}+2\Theta_{6}-2\Theta_{7},\ 
y_7^{L_1}=-\Theta_{1}+2\Theta_{7}.$$
Then,
$$ \Rep^\C_{\poly}(L_1)
=\C[y_1^{L_1}]\otimes_\C\left[R\oplus (y_2^{L_1}y_3^{L_1}y_4^{L_1}y_5^{L_1}y_6^{L_1}y_7^{L_1})R\right],$$
where $R=\C_{\sym}\left[(y_2^{L_1})^2, (y_3^{L_1})^2, (y_4^{L_1})^2 , (y_5^{L_1})^2, (y_6^{L_1})^2, (y_7^{L_1})^2 \right].$\\
Further, the generators of $\Rep^\C_{\poly}(L_1)$ under $\xi^{P_1}$ go to 
$$y_1^{L_1}\mapsto \epsilon_{s_1}^{P_1},$$
$$e_1\left((y_2^{L_1})^2, (y_3^{L_1})^2, (y_4^{L_1})^2, (y_5^{L_1})^2, (y_6^{L_1})^2, (y_7^{L_1})^2\right)\mapsto -2\epsilon_{s_3s_1}^{P_1},$$
$$e_2\left((y_2^{L_1})^2, (y_3^{L_1})^2, (y_4^{L_1})^2, (y_5^{L_1})^2, (y_6^{L_1})^2, (y_7^{L_1})^2\right)\mapsto -57\epsilon_{s_2s_4s_3s_1}^{P_1}+39\epsilon_{s_5s_4s_3s_1}^{P_1},$$
$$e_3\left((y_2^{L_1})^2, (y_3^{L_1})^2, (y_4^{L_1})^2, (y_5^{L_1})^2, (y_6^{L_1})^2, (y_7^{L_1})^2\right)\mapsto$$
$$ -120\epsilon_{s_4s_5s_2s_4s_3s_1}^{P_1}+204\epsilon_{s_6s_5s_2s_4s_3s_1}^{P_1}-316\epsilon_{s_7s_6s_5s_4s_3s_1}^{P_1},$$
$$e_4\left((y_2^{L_1})^2, (y_3^{L_1})^2, (y_4^{L_1})^2, (y_5^{L_1})^2, (y_6^{L_1})^2, (y_7^{L_1})^2\right)\mapsto$$
$$-66\epsilon_{s_1s_3s_4s_5s_2s_4s_3s_1}^{P_1}+249\epsilon_{s_6s_3s_4s_5s_2s_4s_3s_1}^{P_1}-453\epsilon_{s_5s_6s_4s_5s_2s_4s_3s_1}^{P_1}+247\epsilon_{s_7s_6s_4s_5s_2s_4s_3s_1}^{P_1},$$
$$e_5\left((y_2^{L_1})^2, (y_3^{L_1})^2, (y_4^{L_1})^2, (y_5^{L_1})^2, (y_6^{L_1})^2, (y_7^{L_1})^2\right)\mapsto$$
$$438\epsilon_{s_5s_6s_1s_3s_4s_5s_2s_4s_3s_1}^{P_1}-1618\epsilon_{s_7s_6s_1s_3s_4s_5s_2s_4s_3s_1}^{P_1}-408\epsilon_{s_4s_5s_6s_3s_4s_5s_2s_4s_3s_1}^{P_1}$$
$$+1132\epsilon_{s_7s_5s_6s_3s_4s_5s_2s_4s_3s_1}^{P_1}-2140\epsilon_{s_6s_7s_5s_6s_4s_5s_2s_4s_3s_1}^{P_1}.$$
$$e_6\left((y_2^{L_1})^2, (y_3^{L_1})^2, (y_4^{L_1})^2, (y_5^{L_1})^2, (y_6^{L_1})^2, (y_7^{L_1})^2\right)\mapsto$$ $$-27\epsilon_{s_2s_4s_5s_6s_1s_3s_4s_5s_2s_4s_3s_1}^{P_1}+153\epsilon_{s_3s_4s_5s_6s_1s_3s_4s_5s_2s_4s_3s_1}^{P_1}-585\epsilon_{s_7s_4s_5s_6s_1s_3s_4s_5s_2s_4s_3s_1}^{P_1}$$
$$+1552\epsilon_{s_6s_7s_5s_6s_1s_3s_4s_5s_2s_4s_3s_1}^{P_1}+498\epsilon_{s_7s_2s_4s_5s_6s_3s_4s_5s_2s_4s_3s_1}^{P_1}-450\epsilon_{s_6s_7s_4s_5s_6s_3s_4s_5s_2s_4s_3s_1}^{P_1}.$$
$$e_6\left(y_2^{L_1}, y_3^{L_1}, y_4^{L_1}, y_5^{L_1}, y_6^{L_1}, y_7^{L_1}\right)\mapsto -6\epsilon_{s_4s_5s_2s_4s_3s_1}^{P_1} + 15\epsilon_{s_6s_5s_2s_4s_3s_1}^{P_1} - 43\epsilon_{s_7s_6s_5s_4s_3s_1}^{P_1},$$

(b) ($L_2$): For $L_2$, we let 
$$y_1^{L_2}=4\Theta_{1}-2\Theta_{2},\ y_2^{L_2}=-4\Theta_{1}-2\Theta_{2}+4\Theta_{3},\ y_3^{L_2}=-2\Theta_{2}-4\Theta_{3}+4\Theta_{4},$$
$$y_4^{L_2}=2\Theta_{2}-4\Theta_{4}+4\Theta_{5},\ y_5^{L_2}=2\Theta_{2}-4\Theta_{5}+4\Theta_{6},\ y_6^{L_2}=2\Theta_{2}-4\Theta_{6}+4\Theta_{7},\ y_7^{L_2}=2\Theta_{2}-4\Theta_{7}.$$
Then,
$$  \Rep^\C_{\poly}(L_2)
=\C_{\sym}[y_1^{L_2},y_2^{L_2}, y_3^{L_2}, y_4^{L_2}, y_5^{L_2}, y_6^{L_2}, y_7^{L_2}].$$
Further, the generators of $\Rep^\C_{\poly}(L_2)$ under $\xi^{P_2}$ go to 
$$e_1\left( y_1^{L_2}, y_2^{L_2}, y_3^{L_2}, y_4^{L_2}, y_5^{L_2}, y_6^{L_2}, y_7^{L_2}\right)\mapsto 2\epsilon_{s_2}^{P_2},$$
$$e_2\left( y_1^{L_2}, y_2^{L_2}, y_3^{L_2}, y_4^{L_2}, y_5^{L_2}, y_6^{L_2}, y_7^{L_2}\right)\mapsto 4\epsilon_{s_4s_2}^{P_2},$$
$$e_3\left( y_1^{L_2}, y_2^{L_2}, y_3^{L_2}, y_4^{L_2}, y_5^{L_2}, y_6^{L_2}, y_7^{L_2}\right)\mapsto 72\epsilon_{s_3s_4s_2}^{P_2} - 56\epsilon_{s_5s_4s_2}^{P_2},$$
$$e_4\left( y_1^{L_2}, y_2^{L_2}, y_3^{L_2}, y_4^{L_2}, y_5^{L_2}, y_6^{L_2}, y_7^{L_2}\right)\mapsto 432\epsilon_{s_1s_3s_4s_2}^{P_2} - 160\epsilon_{s_5s_3s_4s_2}^{P_2} + 176\epsilon_{s_6s_5s_4s_2}^{P_2},$$
$$e_5\left( y_1^{L_2}, y_2^{L_2}, y_3^{L_2}, y_4^{L_2}, y_5^{L_2}, y_6^{L_2}, y_7^{L_2}\right)\mapsto
32\epsilon_{s_5s_1s_3s_4s_2}^{P_2} - 320\epsilon_{s_4s_5s_3s_4s_2}^{P_2} + 544\epsilon_{s_6s_5s_3s_4s_2}^{P_2} - 1184\epsilon_{s_7s_6s_5s_4s_2}^{P_2},$$
$$e_6\left( y_1^{L_2}, y_2^{L_2}, y_3^{L_2}, y_4^{L_2}, y_5^{L_2}, y_6^{L_2}, y_7^{L_2}\right)\mapsto$$
$$-1088\epsilon_{s_4s_5s_1s_3s_4s_2}^{P_2} + 2176\epsilon_{s_6s_5s_1s_3s_4s_2}^{P_2} + 384\epsilon_{s_2s_4s_5s_3s_4s_2}^{P_2} -64\epsilon_{s_6s_4s_5s_3s_4s_2}^{P_2} -1280\epsilon_{s_7s_6s_5s_3s_4s_2}^{P_2}.$$
$$e_7\left( y_1^{L_2}, y_2^{L_2}, y_3^{L_2}, y_4^{L_2}, y_5^{L_2}, y_6^{L_2}, y_7^{L_2}\right)\mapsto$$
$$-384\epsilon_{s_2s_4s_5s_1s_3s_4s_2}^{P_2} -1152\epsilon_{s_3s_4s_5s_1s_3s_4s_2}^{P_2} + 2048\epsilon_{s_6s_4s_5s_1s_3s_4s_2}^{P_2} -8448\epsilon_{s_7s_6s_5s_1s_3s_4s_2}^{P_2}$$
$$-384\epsilon_{s_6s_2s_4s_5s_3s_4s_2}^{P_2}-1152\epsilon_{s_5s_6s_4s_5s_3s_4s_2}^{P_2}+2432\epsilon_{s_7s_6s_4s_5s_3s_4s_2 }^P.$$

(c) ($L_3$): For $L_3$, we let 
$$y_1^{L_3}=3\Theta_{1}-\Theta_{3},\ y_2^{L_3}=-3\Theta_{1}+2\Theta_{3},\ y_3^{L_3}=-4\Theta_{2}+3\Theta_{3},\ y_4^{L_3}=4\Theta_{2}+3\Theta_{3}-4\Theta_{4},$$
$$y_5^{L_3}=-\Theta_{3}+4\Theta_{4}-4\Theta_{5},\ y_6^{L_3}=-\Theta_{3}+4\Theta_{5}-4\Theta_{6},\ y_7^{L_3}=-\Theta_{3}+4\Theta_{6}-4\Theta_{7},\ y_8^{L_3}=-\Theta_{3}+4\Theta_{7}.$$
Then,
$$ \Rep^\C_{\poly}(L_3)
=\C_{\sym}[y_1^{L_3},y_2^{L_3}] \otimes_\C \C_{\sym}[y_3^{L_3}, y_4^{L_3}, y_5^{L_3}, y_6^{L_3}, y_7^{L_3}, y_8^{L_3}]/R_{7,3},$$
where $R_{7,3}=2y_1^{L_3}+2y_2^{L_3}-y_3^{L_3}- y_4^{L_3}- y_5^{L_3}- y_6^{L_3}- y_7^{L_3}- y_8^{L_3}$.\\
Further, the generators of $\Rep^\C_{\poly}(L_3)$ under $\xi^{P_3}$ go to 
$$e_1\left( y_1^{L_3}, y_2^{L_3}\right)\mapsto \epsilon_{s_3}^{P_3},$$
$$e_2\left( y_1^{L_3}, y_2^{L_3}\right)\mapsto 7\epsilon_{s_1s_3}^{P_3} - 2\epsilon_{s_4s_3}^{P_3},$$
$$e_1\left( y_3^{L_3}, y_4^{L_3}, y_5^{L_3}, y_6^{L_3}, y_7^{L_3}, y_8^{L_3}\right)\mapsto 2\epsilon_{s_3}^{P_3},$$
$$e_2\left( y_3^{L_3}, y_4^{L_3}, y_5^{L_3}, y_6^{L_3}, y_7^{L_3}, y_8^{L_3}\right)\mapsto -9\epsilon_{s_1s_3}^{P_3} +7\epsilon_{s_4s_3}^{P_3},$$
$$e_3\left( y_3^{L_3}, y_4^{L_3}, y_5^{L_3}, y_6^{L_3}, y_7^{L_3}, y_8^{L_3}\right)\to -8\epsilon_{s_4s_1s_3}^{P_3} -68\epsilon_{s_2s_4s_3}^{P_3} +60\epsilon_{s_5s_4s_3}^{P_3},$$
$$e_4\left( y_3^{L_3}, y_4^{L_3}, y_5^{L_3}, y_6^{L_3}, y_7^{L_3}, y_8^{L_3}\right)\mapsto$$
$$93\epsilon_{s_2s_4s_1s_3}^{P_3} - 34\epsilon_{s_3s_4s_1s_3}^{P_3} - 35\epsilon_{s_5s_4s_1s_3}^{P_3} -130\epsilon_{s_5s_2s_4s_3}^{P_3} +255\epsilon_{s_6s_5s_4s_3}^{P_3}.$$
$$e_5\left( y_3^{L_3}, y_4^{L_3}, y_5^{L_3}, y_6^{L_3}, y_7^{L_3}, y_8^{L_3}\right)\mapsto$$
$$42\epsilon_{s_3s_2s_4s_1s_3}^{P_3} + 16\epsilon_{s_5s_2s_4s_1s_3}^{P_3} - 86\epsilon_{s_5s_3s_4s_1s_3}^{P_3} + 136\epsilon_{s_6s_5s_4s_1s_3}^{P_3}+196\epsilon_{s_4s_5s_2s_4s_3}^{P_3} $$ $$- 538\epsilon_{s_6s_5s_2s_4s_3}^{P_3} + 1314\epsilon_{s_7s_6s_5s_4s_3}^{P_3}.$$
$$e_6\left( y_3^{L_3}, y_4^{L_3}, y_5^{L_3}, y_6^{L_3}, y_7^{L_3}, y_8^{L_3}\right)\mapsto$$
$$-35\epsilon_{s_4s_3s_2s_4s_1s_3}^{P_3} + 98\epsilon_{s_5s_3s_2s_4s_1s_3}^{P_3} - 38\epsilon_{s_4s_5s_2s_4s_1s_3}^{P_3} - 121\epsilon_{s_6s_5s_2s_4s_1s_3}^{P_3}$$
$$+93\epsilon_{s_4s_5s_3s_4s_1s_3}^{P_3} - 399\epsilon_{s_6s_5s_3s_4s_1s_3}^{P_3} + 1965\epsilon_{s_7s_6s_5s_4s_1s_3}^{P_3}$$
$$-78\epsilon_{s_3s_4s_5s_2s_4s_3}^{P_3} + 253\epsilon_{s_6s_4s_5s_2s_4s_3}^{P_3} - 1308\epsilon_{s_7s_6s_5s_2s_4s_3}^{P_3}.$$

(d) ($L_4$): For $L_4$, we let 
$$y_1^{L_4}=4\Theta_{1}-\Theta_{4},\ y_2^{L_4}=-4\Theta_{1}+4\Theta_{3}-\Theta_{4},\ y_3^{L_4}=-4\Theta_{3}+3\Theta_{4},$$
$$y_4^{L_4}=-\Theta_{4}+6\Theta_{7},\ y_5^{L_4}=-\Theta_{4}+6\Theta_{6}-6\Theta_{7},\ y_6^{L_4}=-\Theta_{4}+6\Theta_{5}-6\Theta_{6},$$
$$y_7^{L_4}=5\Theta_{4}-6\Theta_{5}, y_8^{L_4}=12\Theta_{2}-5\Theta_{4}, y_9^{L_4}=-12\Theta_{2}+7\Theta_{4}.$$
Then,
$$  \Rep^\C_{\poly}(L_4)
=\C_{\sym}[y_1^{L_4},y_2^{L_4},y_3^{L_4}] \otimes_\C \C_{\sym}[y_4^{L_4}, y_5^{L_4}, y_6^{L_4}, y_7^{L_4}] \otimes_\C \C_{\sym}[y_8^{L_4}, y_9^{L_4}]/(R_{7,4},R_{7,4}'),$$
where $R_{7,4}=y_1^{L_4}+y_2^{L_4}+y_3^{L_4}-y_8^{L_4}- y_9^{L_4}$ and $R_{7,4}'=y_4^{L_4}+ y_5^{L_4}+ y_6^{L_4}+ y_7^{L_4}-y_8^{L_4}- y_9^{L_4}$.\\
Further, the generators of $\Rep^\C_{\poly}(L_4)$ under $\xi^{P_4}$ go to
$$e_1\left( y_1^{L_4}, y_2^{L_4}, y_3^{L_4}\right)\mapsto \epsilon_{s_4}^{P_4},$$
$$e_2\left( y_1^{L_4}, y_2^{L_4}, y_3^{L_4}\right)\mapsto -5\epsilon_{s_2s_4}^{P_4} + 11\epsilon_{s_3s_4}^{P_4} - 5\epsilon_{s_5s_4}^{P_4},$$
$$e_3\left( y_1^{L_4}, y_2^{L_4}, y_3^{L_4}\right)\mapsto
-10\epsilon_{s_3s_2s_4}^{P_4} + 6\epsilon_{s_5s_2s_4}^{P_4} + 51\epsilon_{s_1s_3s_4}^{P_4} -10\epsilon_{s_5s_3s_4}^{P_4} + 3\epsilon_{s_6s_5s_4}^{P_4},$$
$$e_1\left( y_4^{L_4}, y_5^{L_4}, y_6^{L_4}, y_7^{L_4}\right)\mapsto 2\epsilon_{s_4}^{P_4},$$
$$e_2\left( y_4^{L_4}, y_5^{L_4}, y_6^{L_4}, y_7^{L_4}\right)\mapsto -12\epsilon_{s_2s_4}^{P_4} - 12\epsilon_{s_3s_4}^{P_4} 24\epsilon_{s_5s_4}^{P_4},$$
$$e_3\left( y_4^{L_4}, y_5^{L_4}, y_6^{L_4}, y_7^{L_4}\right)\mapsto
28\epsilon_{s_3s_2s_4}^{P_4} - 44\epsilon_{s_5s_2s_4}^{P_4} + 14\epsilon_{s_1s_3s_4}^{P_4} - 44\epsilon_{s_5s_3s_4}^{P_4} + 158\epsilon_{s_6s_5s_4}^{P_4},$$
$$e_4\left( y_4^{L_4}, y_5^{L_4}, y_6^{L_4}, y_7^{L_4}\right)\mapsto$$
$$-15\epsilon_{s_1s_3s_2s_4}^{P_4} - 10\epsilon_{s_4s_3s_2s_4}^{P_4} + 42\epsilon_{s_5s_3s_2s_4}^{P_4} + 26\epsilon_{s_4s_5s_2s_4}^{P_4} - 159\epsilon_{s_6s_5s_2s_4}^{P_4}$$
$$+21\epsilon_{s_5s_1s_3s_4}^{P_4} + 26\epsilon_{s_4s_5s_3s_4}^{P_4} - 159\epsilon_{s_6s_5s_3s_4}^{P_4} + 1111\epsilon_{s_7s_6s_5s_4}^{P_4},$$
$$e_1\left( y_8^{L_4}, y_9^{L_4}\right)\mapsto 2\epsilon_{s_4}^{P_4},$$
$$e_2\left( y_8^{L_4}, y_9^{L_4}\right)\mapsto 109\epsilon_{s_2s_4}^{P_4} - 35\epsilon_{s_3s_4}^{P_4} - 35\epsilon_{s_5s_4}^{P_4}.$$

(e) ($L_5$): For $L_5$, we let 
$$y_1^{L_5}=6\Theta_{1}-2\Theta_{5},\ y_2^{L_5}=-6\Theta_{1}+6\Theta_{3}-2\Theta_{5},\ y_3^{L_5}=-6\Theta_{3}+6\Theta_{4}-2\Theta_{5},\ y_4^{L_5}=6\Theta_{2}-6\Theta_{4}+4\Theta_{5},$$
$$y_5^{L_5}=-6\Theta_{2}+4\Theta_{5},\ y_6^{L_5}=-\Theta_{5}+5\Theta_{7},\ y_7^{L_5}=-\Theta_{5}+5\Theta_{6}-5\Theta_{7},\ y_8^{L_5}=4\Theta_{5}-5\Theta_{6}.$$
Then,
$$  \Rep^\C_{\poly}(L_5)
=\C_{\sym}[y_1^{L_5},y_2^{L_5},y_3^{L_5}, y_4^{L_5}, y_5^{L_5}] \otimes_\C \C_{\sym}[ y_6^{L_5}, y_7^{L_5}, y_8^{L_5}]/R_{7,5},$$
where $R_{7,5}=y_1^{L_5}+y_2^{L_5}+y_3^{L_5}+ y_4^{L_5}+ y_5^{L_5} -  y_6^{L_5}- y_7^{L_5}- y_8^{L_5}$.\\
Further, the generators of $\Rep^\C_{\poly}(L_5)$ under $\xi^{P_5}$ go to
$$e_1\left( y_1^{L_5}, y_2^{L_5},y_3^{L_5}, y_4^{L_5}, y_5^{L_5}\right)\mapsto 2\epsilon_{s_5}^{P_5},$$
$$e_2\left( y_1^{L_5}, y_2^{L_5},y_3^{L_5}, y_4^{L_5}, y_5^{L_5}\right)\mapsto 16\epsilon_{s_4s_5}^{P_5} - 20\epsilon_{s_6s_5}^{P_5},$$
$$e_3\left( y_1^{L_5}, y_2^{L_5},y_3^{L_5}, y_4^{L_5}, y_5^{L_5}\right)\mapsto - 224\epsilon_{s_2s_4s_5}^{P_5} + 208\epsilon_{s_3s_4s_5}^{P_5} - 16\epsilon_{s_6s_4s_5}^{P_5} - 8\epsilon_{s_7s_6s_5}^{P_5},$$
$$e_4\left( y_1^{L_5}, y_2^{L_5},y_3^{L_5}, y_4^{L_5}, y_5^{L_5}\right)\mapsto$$
$$-608\epsilon_{s_3s_2s_4s_5}^{P_5} + 384\epsilon_{s_6s_2s_4s_5}^{P_5} + 1424\epsilon_{s_1s_3s_4s_5}^{P_5} - 48\epsilon_{s_6s_3s_4s_5}^{P_5} - 176\epsilon_{s_5s_6s_4s_5}^{P_5} - 48\epsilon_{s_7s_6s_4s_5}^{P_5},$$
$$e_5\left( y_1^{L_5}, y_2^{L_5},y_3^{L_5}, y_4^{L_5}, y_5^{L_5}\right)\mapsto$$
$$-2976\epsilon_{s_1s_3s_2s_4s_5}^{P_5} + 896\epsilon_{s_4s_3s_2s_4s_5}^{P_5} - 160\epsilon_{s_6s_3s_2s_4s_5}^{P_5} + 224\epsilon_{s_5s_6s_2s_4s_5}^{P_5} + 96\epsilon_{s_7s_6s_2s_4s_5}^{P_5}$$ $$+2944\epsilon_{s_6s_1s_3s_4s_5}^{P_5} - 640\epsilon_{s_5s_6s_3s_4s_5}^{P_5} - 768\epsilon_{s_7s_6s_3s_4s_5}^{P_5} + 512\epsilon_{s_7s_5s_6s_4s_5}^{P_5},$$
$$e_1\left( y_6^{L_5}, y_7^{L_5}, y_8^{L_5}\right)\mapsto 2\epsilon_{s_5}^{P_5},$$
$$e_2\left( y_6^{L_5}, y_7^{L_5}, y_8^{L_5}\right)\mapsto -7\epsilon_{s_4s_5}^{P_5} + 18\epsilon_{s_6s_5}^{P_5},$$
$$e_3\left( y_6^{L_5}, y_7^{L_5}, y_8^{L_5}\right)\mapsto 4\epsilon_{s_2s_4s_5}^{P_5} + 4\epsilon_{s_3s_4s_5}^{P_5} - 17\epsilon_{s_6s_4s_5}^{P_5} + 104\epsilon_{s_7s_6s_5}^{P_5},$$

(f) ($L_6$): For $L_6$, we let 
$$y_1^{L_6}=2\Theta_{1}-\Theta_{6},\ y_2^{L_6}=-2\Theta_{1}+2\Theta_{3}-\Theta_{6},\ y_3^{L_6}=-2\Theta_{3}+2\Theta_{4}-\Theta_{6},\ y_4^{L_6}=2\Theta_{2}-2\Theta_{4}+2\Theta_{5}-\Theta_{6},$$
$$y_5^{L_6}=-2\Theta_{2}+2\Theta_{5}-\Theta_{6},\ y_6^{L_6}=-\Theta_{6}+4\Theta_{7},\ 
y_7^{L_6}=3\Theta_{6}-4\Theta_{7}.$$
Then,
$$ \Rep^\C_{\poly}(L_6)
=\left[R\oplus (y_1^{L_6}y_2^{L_6}y_3^{L_6}y_4^{L_6}y_5^{L_6})R\right]\otimes_\C\C[y_6^{L_6},y_7^{L_6}],$$
where $R=\C_{\sym}\left[(y_1^{L_6})^2, (y_2^{L_6})^2, (y_3^{L_6})^2, (y_4^{L_6})^2 , (y_5^{L_6})^2 \right].$\\
Further, the generators of $\Rep^\C_{\poly}(L_6)$ under $\xi^{P_6}$ go to
$$e_1\left((y_1^{L_6})^2, (y_2^{L_6})^2, (y_3^{L_6})^2, (y_4^{L_6})^2 , (y_5^{L_6})^2\right)\mapsto -3\epsilon_{s_5s_6}^{P_6} + 5\epsilon_{s_7s_6}^{P_6},$$
$$e_2\left((y_1^{L_6})^2, (y_2^{L_6})^2, (y_3^{L_6})^2, (y_4^{L_6})^2 , (y_5^{L_6})^2\right)\mapsto -54\epsilon_{s_2s_4s_5s_6}^{P_6} + 42\epsilon_{s_3s_4s_5s_6}^{P_6} - 2\epsilon_{s_7s_4s_5s_6}^{P_6} + 4\epsilon_{s_6s_7s_5s_6}^{P_6},$$
$$e_3\left((y_1^{L_6})^2, (y_2^{L_6})^2, (y_3^{L_6})^2, (y_4^{L_6})^2 , (y_5^{L_6})^2\right)\mapsto$$
$$174\epsilon_{s_1s_3s_2s_4s_5s_6}^{P_6} - 108\epsilon_{s_4s_3s_2s_4s_5s_6}^{P_6} + 100\epsilon_{s_7s_3s_2s_4s_5s_6}^{P_6} - 54\epsilon_{s_6s_7s_2s_4s_5s_6}^{P_6}$$ $$- 174\epsilon_{s_7s_1s_3s_4s_5s_6}^{P_6} - 22\epsilon_{s_6s_7s_3s_4s_5s_6}^{P_6} + 50\epsilon_{s_5s_6s_7s_4s_5s_6}^{P_6},$$
$$e_4\left((y_1^{L_6})^2, (y_2^{L_6})^2, (y_3^{L_6})^2, (y_4^{L_6})^2 , (y_5^{L_6})^2\right)\mapsto$$
$$-519\epsilon_{s_3s_4s_1s_3s_2s_4s_5s_6}^{P_6}+291\epsilon_{s_5s_4s_1s_3s_2s_4s_5s_6}^{P_6} +207\epsilon_{s_7s_4s_1s_3s_2s_4s_5s_6}^{P_6} -164\epsilon_{s_6s_7s_1s_3s_2s_4s_5s_6}^{P_6}$$
$$+42\epsilon_{s_6s_5s_4s_3s_2s_4s_5s_6}^{P_6} -250\epsilon_{s_7s_5s_4s_3s_2s_4s_5s_6}^{P_6} +168\epsilon_{s_6s_7s_4s_3s_2s_4s_5s_6}^{P_6} -136\epsilon_{s_5s_6s_7s_3s_2s_4s_5s_6}^{P_6}$$
$$+70\epsilon_{s_4s_5s_6s_7s_2s_4s_5s_6}^{P_6} -4\epsilon_{s_5s_6s_7s_1s_3s_4s_5s_6}^{P_6} +102\epsilon_{s_4s_5s_6s_7s_3s_4s_5s_6}^{P_6}.$$
$$e_5\left((y_1^{L_6})^2, (y_2^{L_6})^2, (y_3^{L_6})^2, (y_4^{L_6})^2 , (y_5^{L_6})^2\right)\mapsto$$
$$-180\epsilon_{s_4s_5s_3s_4s_1s_3s_2s_4s_5s_6}^{P_6}
+333\epsilon_{s_6s_5s_3s_4s_1s_3s_2s_4s_5s_6}^{P_6}
-84\epsilon_{s_7s_5s_3s_4s_1s_3s_2s_4s_5s_6}^{P_6}
+223\epsilon_{s_6s_7s_3s_4s_1s_3s_2s_4s_5s_6}^{P_6}$$
$$-444\epsilon_{s_7s_6s_5s_4s_1s_3s_2s_4s_5s_6}^{P_6}
-106\epsilon_{s_6s_7s_5s_4s_1s_3s_2s_4s_5s_6}^{P_6}
-76\epsilon_{s_5s_6s_7s_4s_1s_3s_2s_4s_5s_6}^{P_6}
-131\epsilon_{s_4s_5s_6s_7s_1s_3s_2s_4s_5s_6}^{P_6}$$
$$+152\epsilon_{s_6s_7s_6s_5s_4s_3s_2s_4s_5s_6}^{P_6}
+12\epsilon_{s_5s_6s_7s_5s_4s_3s_2s_4s_5s_6}^{P_6}
+26\epsilon_{s_4s_5s_6s_7s_4s_3s_2s_4s_5s_6}^{P_6}
+136\epsilon_{s_2s_4s_5s_6s_7s_3s_2s_4s_5s_6}^{P_6}$$
$$-56\epsilon_{s_3s_4s_5s_6s_7s_3s_2s_4s_5s_6}^{P_6}
+264\epsilon_{s_1s_3s_4s_5s_6s_7s_2s_4s_5s_6}^{P_6}
-504\epsilon_{s_2s_4s_5s_6s_7s_1s_3s_4s_5s_6}^{P_6}
+506\epsilon_{s_3s_4s_5s_6s_7s_1s_3s_4s_5s_6}^{P_6}.$$
$$e_5\left(y_1^{L_6}, y_2^{L_6}, y_3^{L_6}, y_4^{L_6}, y_5^{L_6}\right)\mapsto$$
$$6\epsilon_{s_3s_2s_4s_5s_6}^{P_6} -8\epsilon_{s_7s_2s_4s_5s_6}^{P_6} -21\epsilon_{s_1s_3s_4s_5s_6}^{P_6} +8\epsilon_{s_7s_3s_4s_5s_6}^{P_6} -1\epsilon_{s_6s_7s_4s_5s_6}^{P_6},$$
and 
$$e_1\left(y_6^{L_6}, y_7^{L_6}\right) \mapsto  2\epsilon_{s_6}^{P_6},$$
$$e_2\left(y_6^{L_6}, y_7^{L_6}\right) \mapsto  -3\epsilon_{s_5s_6}^{P_6} +13\epsilon_{s_7s_6}^{P_6}.$$

(g) ($L_7$): 
For $E_6$, we let
\begin{align*}
    x_1  &= 5\alpha_1 + 4\alpha_3 + 3\alpha_4 + 2\alpha_5 + \alpha_6,\,\,
    x_2  = -\alpha_1 + 4\alpha_3 + 3\alpha_4 + 2\alpha_5 + \alpha_6,\\
    x_3  &= -\alpha_1 - 2\alpha_3 + 3\alpha_4 + 2\alpha_5 + \alpha_6,\,\,
    x_4  = -\alpha_1 - 2\alpha_3 - 3\alpha_4 + 2\alpha_5 + \alpha_6,\\
    x_5  &= -\alpha_1 - 2\alpha_3 - 3\alpha_4 - 4\alpha_5 + \alpha_6,\,,
    x_6  = -\alpha_1 - 2\alpha_3 - 3\alpha_4 - 4\alpha_5 - 5\alpha_6\\
\end{align*}
and
$$x= -3(\alpha_1 + 2\alpha_2 + 2\alpha_3 + 3\alpha_4 + 2\alpha_5 + \alpha_6).$$
Set
\begin{align*}
    a_i& =x_i+x,\ b_i=x_i-x\ (i=1,\dots,6),\\
    c_{ij}&= -x_i-x_j\ (1\le i<j\le 6).
\end{align*}
Then,  the following provides a set of generators for the Weyl group invariant polynomials in $S(\ft_6^*)=\C[\alpha_1, \dots, \alpha_6]$ for $E_6$, where $\ft_6$ is the Cartan subalgebra of $E_6$
 (cf. \cite[$\S$3]{lee}, \cite{C}): 
\begin{equation}\label{E6generator}
    \psi_m=\sum_{i=1}^6a_i^m+\sum_{i=1}^6b_i^m+\sum_{1\leq i<j\leq 6}c_{ij}^m,\ m=2,5,6,8,9,12.
\end{equation}
We view $\psi_m$ as elements of $S(\ft^*) = \C[\alpha_1, \dots, \alpha_6, \alpha_7]\supset \C[\alpha_1, \dots, \alpha_6]$.
Then, 
$$S(\ft^*)^{W_7} = \C[\psi_2,\psi_5,\psi_6,\psi_8,\psi_9,\psi_{12}]\otimes_\C \C[\omega_7],$$
where $W_7$ is the Weyl group of $L_7$.
Thus, 
$$\Rep^\C_{\poly}(L_7)=\C[\bar{\psi}_2,\bar{\psi}_5,\bar{\psi}_6,\bar{\psi}_8,\bar{\psi}_9,\bar{\psi}_{12}]\otimes_\C \C[\bar{\omega}_7],$$
 where $\bar{\psi}_m := \theta_{\omega_7}^*(\psi_m)$ and $\bar{\omega}_7:= \theta_{\omega_7}^*(\omega_7)$ 
 ($\theta_{\omega_7}^*: S(\ft^*)^{W_7} \to  \Rep^\C_{\poly}(L_7)$  being the induced map from the Springer morphism $\theta_{\omega_7}: T \to \ft$). 
 
  Further, the generators of $\Rep^\C_{\poly}(L_7)$  
   under $\xi^{P_7}$ go to
$$\frac{1}{144}\bar{\psi}_2 \mapsto -\epsilon_{s_6s_7}^{P_7},$$
$$\frac{1}{5760}\bar{\psi}_5 \mapsto -92\epsilon_{s_2s_4s_5s_6s_7}^{P_7} + 70\epsilon_{s_3s_4s_5s_6s_7}^{P_7},$$
$$\frac{1}{3456}\bar{\psi}_6 \mapsto -508\epsilon_{s_3s_2s_4s_5s_6s_7}^{P_7} + 1366\epsilon_{s_1s_3s_4s_5s_6s_7}^{P_7},$$
$$\frac{1}{9216}\bar{\psi}_8 \mapsto -20165\epsilon_{s_4s_1s_3s_2s_4s_5s_6s_7}^{P_7} + 23686\epsilon_{s_5s_4s_3s_2s_4s_5s_6s_7}^{P_7},$$
$$\frac{1}{387072}\bar{\psi}_9 \mapsto -17632\epsilon_{s_3s_4s_1s_3s_2s_4s_5s_6s_7}^{P_7} + 14584\epsilon_{s_5s_4s_1s_3s_2s_4s_5s_6s_7}^{P_7} -7150\epsilon_{s_6s_5s_4s_3s_2s_4s_5s_6s_7}^{P_7},$$
\begin{align*}
    \frac{1}{663552}\bar{\psi}_{12} & \mapsto -6513792\epsilon_{s_2s_4s_5s_3s_4s_1s_3s_2s_4s_5s_6s_7}^{P_7} + 2660262 \epsilon_{s_6s_4s_5s_3s_4s_1s_3s_2s_4s_5s_6s_7}^{P_7}\\
    & -1961752\epsilon_{s_7s_6s_5s_3s_4s_1s_3s_2s_4s_5s_6s_7}^{P_7}.
\end{align*}
and $$\bar{\omega}_7\mapsto \epsilon_{s_7}^{P_7}$$
\end{theorem}

    \bibliographystyle{plain}

\begin{thebibliography}{10}

\bibitem[BR] {BR} P. Bardsley and R.W. Richardson, \'{E}tale slices for algebraic transformation groups in characteristic $p$,
{\em Proc. London Math. Soc.} {\bf 51} (1985), 295--317.

\bibitem[BL]{BL} S. Billey and V. Lakshmibai, {\it Singular Loci of Schubert Varieties}, Progress in Mathematics, Vol. {\bf 182}, Birkh\"auser, 2000.

\bibitem[Bo]{Bo}
N. Bourbaki, {\em Groupes et Alg\`ebres de Lie}, Chap. 4--6,
Masson, Paris, 1981.

\bibitem[BKT1] {BKT1} A. Buch, A. Kresch and H. Tamvakis, Quantum Pieri rules for isotropic Grassmannians, {\it Invent.  Math.} {\bf 178} (2009), 345--405.

\bibitem[BKT2] {BKT2} A. Buch, A. Kresch and H. Tamvakis, Quantum  Giambelli formulas  for isotropic Grassmannians, {\it  Math. Annalen} {\bf 354} (2012), 801--812.

\bibitem[C]{C} H. Coxeter, The product of the generators of a finite group generated by reflections, {\it Duke
Math. Journal} {\bf 18} (1951), 765--782.

\bibitem[D] {D} H. Duan, Multiplicative rule of Schubert classes, {\it Invent.  Math.} {\bf 159} (2005), 407--436.


\bibitem[F]{fulton1}
W. Fulton,
{\em Young Tableaux},
 London Math. Society, Cambridge University Press, 1997.

\bibitem[Ku1]{Ku1}
S.  Kumar, {\em Kac-Moody Groups, their Flag Varieties and
Representation Theory}, Progress in Mathematics, vol. {\bf 204},
Birkh\"auser, 2002.

\bibitem[Ku2]{Ku2}
S.  Kumar, 
Representation ring of Levi subgroups versus cohomology ring of flag varieties,
{\it Math. Annalen} {\bf 366} (2016), 395 --415. 

\bibitem[Ku3]{Ku3}
S.  Kumar, 
{\em Conformal Blocks, Generalized Theta Functions and the Verlinde Formula}, New Mathematical Monographs vol. 42,
Cambridge University Press, Cambridge, 2022.
 
\bibitem[KR] {KR} S. Kumar and S. Rogers, Representation ring of Levi subgroups versus cohomology ring of flag varieties II, {\it J. of Algebra} {\bf 556} (2020), 340--362. 

\bibitem[Lee]{blee} B. Lee, {\it Comparison of eigencones under certain diagram automorphisms}, PhD thesis, The University of North Carolina at Chapel Hill, 2012.

\bibitem[L]{lee} C. Y. Lee, Invariant polynomials of Weyl groups and applications to the centres of universal enveloping algebras, {\it Canad. J. Math.} {\bf 26} (1974), 583--592.

\bibitem[R]{R} S. Rogers, An explicit determination of the Springer morphism, 
  {\it Communications in Algebra} {\bf 46} (2017), 4233-- 4242.
  
\bibitem[Sp]{Sp}
E.  H. Spanier, {\em Algebraic Topology}, McGraw-Hill, 1966.

\bibitem[X]{X} J. Xie, Exceptional$\cdot$ipynb (2022)
\begin{verbatim}
https://colab.research.google.com/drive/1WdrZseCHoRlWADOng7SwfefXqgbmz3CL
\end{verbatim}

\end{thebibliography}
\def\noopsort#1{}

\vskip5ex

\noindent
Address: Shrawan Kumar,
Department of Mathematics,
University of North Carolina,
Chapel Hill, NC  27599--3250. 
\noindent
email: shrawan@email.unc.edu
\vskip1ex

Jiale Xie,
Department of Mathematics,
University of North Carolina,
Chapel Hill, NC  27599--3250. 
\noindent
email: caleb89@live.unc.edu

\end{document}